\documentclass[conference]{IEEEtran}
\usepackage{times}

\usepackage[hyphens]{url}
\usepackage[numbers]{natbib}
\usepackage{multicol}
\usepackage[bookmarks=true]{hyperref}

\usepackage{hyperref}
\hypersetup{colorlinks = true,
            linkcolor = blue,
            urlcolor  = blue,
            citecolor = blue,
            anchorcolor = blue}

\usepackage{amsmath}
\usepackage{amsthm}
\usepackage{physics}
\usepackage{amssymb}
\usepackage{mathtools}
\usepackage[ruled, vlined, linesnumbered]{algorithm2e}
\usepackage[capitalise]{cleveref}
\usepackage{easyReview}
\usepackage{cancel}
\usepackage{siunitx}
\DeclareSIUnit\ft{ft}
\usepackage{graphicx}
\usepackage{float}
\usepackage{caption}
\usepackage{subcaption}

\newtheorem{theorem}{Theorem}
\newtheorem{proposition}{Proposition}
\newtheorem{lemma}{Lemma}

\theoremstyle{remark}
\newtheorem{remark}{Remark}

\usepackage[nolist]{acronym}
\begin{acronym}
    \acro{DDP}{Differential Dynamic Programming}
    \acro{PDDP}{Parameterized Differential Dynamic Programming}
    \acro{STO}{switching time optimization}
    \acro{MLE}{maximum likelihood estimation}
    \acro{iLQR}{iterative Linear-Quadratic Regulator}
    \acro{MHE}{moving horizon estimation}
    \acro{MPC}{model predictive control}
    \acro{MLE}{maximum likelihood estimation}
    \acro{UAM}{urban air mobility}
    \acro{VTOL}{vertical takeoff and landing}
    \acro{RVLT}{Revolutionary Vertical Lift Technology}
    \acro{STO-DDP}{Switching Time Optimization Differential Dynamic Programming}
\end{acronym}

\newcommand{\btheta}{\boldsymbol{\theta}}
\newcommand{\R}{\mathbb{R}}
\newcommand{\defeq}{\vcentcolon=}

\DeclareMathOperator*{\argmin}{argmin}

\begin{document}

\title{Parameterized Differential Dynamic Programming}




%
\author{
   \authorblockN{
      Alex Oshin\authorrefmark{1}\authorrefmark{2}\authorrefmark{3},
      Matthew D. Houghton\authorrefmark{2},
      Michael J. Acheson\authorrefmark{2},
      Irene M. Gregory\authorrefmark{2} and
      Evangelos A. Theodorou\authorrefmark{1}
   }
   \authorblockA{
      \authorrefmark{1}School of Aerospace Engineering, Georgia Institute of Technology, Atlanta, GA
   }
   \authorblockA{
      \authorrefmark{2}NASA Langley Research Center, Hampton, VA
   }
   \authorblockA{
      \authorrefmark{3}Correspondence to: alexoshin@gatech.edu
   }
}

\maketitle

\begin{abstract}
Differential Dynamic Programming (DDP) is an efficient trajectory optimization algorithm relying on second-order approximations of a system's dynamics and cost function, and has recently been applied to optimize systems with time-invariant parameters. Prior works include system parameter estimation and identifying the optimal switching time between modes of hybrid dynamical systems. This paper generalizes previous work by proposing a general parameterized optimal control objective and deriving a parametric version of DDP, titled Parameterized Differential Dynamic Programming (PDDP). A rigorous convergence analysis of the algorithm is provided, and PDDP is shown to converge to a minimum of the cost regardless of initialization. The effects of varying the optimization to more effectively escape local minima are analyzed. Experiments are presented applying PDDP on multiple robotics systems to solve model predictive control (MPC) and moving horizon estimation (MHE) tasks simultaneously. Finally, PDDP is used to determine the optimal transition point between flight regimes of a complex urban air mobility (UAM) class vehicle exhibiting multiple phases of flight.
\end{abstract}

\IEEEpeerreviewmaketitle

\section{Introduction}
\label{sec/introduction}




Classically, optimal control research has focused on the study of methods for trajectory optimization of underactuated nonlinear systems.
This attention is due to the fact that underactuated systems are more difficult to control due to having more degrees of freedom relative to the number of controls available \cite{spong1998underactuated}.
However, the emerging \acf{UAM} sector allows the opportunity to study the application of optimal control methods to overactuated aircraft whose complexity appears in the nonlinear transition dynamics between flight phases \cite{gregory2021intelligent,gregory2021urban}.
These \ac{UAM} class vehicles commonly exhibit three phases of flight --- including \acf{VTOL}, fixed-wing cruise, and a complex transition phase between the two --- and can thus be classified as hybrid systems \cite{branicky1998multiple,ezzine1989controllability,hedlund1999optimal}, transitioning between multiple modes of the dynamics during a flight.
While many algorithms have been developed for hybrid systems trajectory optimization in the past, most adopt a linear or reduced-order model, with the focus on bipedal or quadrupedal robotics systems, e.g. \cite{orin2013centroidal,apgar2018fast,di2018dynamic}.
The lack of applications to \ac{UAM} class vehicles reveals an opportunity to develop hybrid systems optimization techniques for unique aircraft dynamics which have not been previously studied.

Past research has shown \acf{DDP} is an effective algorithm for planning in high-dimensional state spaces \cite{tassa2014control}, and \ac{DDP} has shown recent success planning trajectories for \ac{UAM} vehicles \cite{houghton2022path}.
\ac{DDP} is a shooting method that achieves computational efficiency using second-order approximations along a nominal trajectory and admits quadratic convergence properties under mild assumptions \cite{mayne1966second,jacobson1970differential}.
\ac{DDP} is beneficial in that it requires no reduction in the complexity of the dynamics model, and its convergence properties make it a solid candidate for solving realtime \ac{MPC} tasks \cite{tassa2012synthesis}.

\begin{figure}[t]
	\centering
	\includegraphics[scale=1.0]{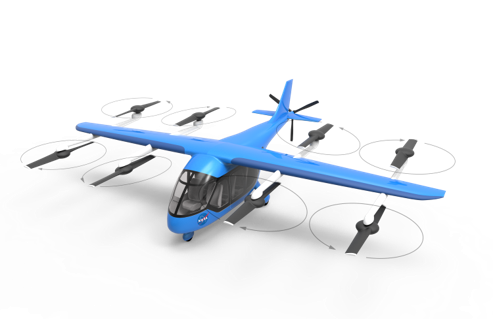}
	\caption{Rendering of the NASA Lift+Cruise aircraft.}
	\label{fig/lpc}
\end{figure}

Recent work has shown the value function derivatives produced by \ac{DDP} can be used to update the initial conditions of the nominal trajectory, which allows optimization over time-invariant parameters \cite{kobilarov2015differential,li2020hybrid}.
While this update corresponds to a Newton step on the parameters, modifying the initial condition of \ac{DDP} significantly affects the convergence of the algorithm since the change induces a large shift of the nominal trajectory in the state space.
To remedy this issue, this work introduces a general parameterized optimal control objective.
Explicitly introducing the parameters into the system dynamics and cost function allows a second-order algorithm to be derived for iteratively updating both the controls and parameters simultaneously, independent of the form of the parameterization. This parameterized version of \ac{DDP} is referred to as \ac{PDDP}, and a rigorous convergence analysis is provided showing the derived algorithm converges to a minimum of the cost regardless of state or parameter initialization.
Further, the effects of the simultaneous optimization of parameters and controls is studied.
An optimization scheme is proposed that more effectively escapes local minima, which is a common issue of local methods such as \ac{DDP}.
The \ac{PDDP} algorithm is applied to two important robotics tasks, model parameter estimation and hybrid systems optimization, and is used to identify the optimal transition points between flight regimes for the NASA Lift+Cruise \ac{UAM} class vehicle developed under NASA's \acf{RVLT} project \cite{silva2018vtol}.
A rendering of the Lift+Cruise aircraft is provided in \cref{fig/lpc}.

The contributions of this work are as follows:
\begin{enumerate}
    \item Generalize previous work by deriving a parameterized form of \ac{DDP}, referred to as \ac{PDDP}.
    \item Provide theoretical analysis of the convergence behavior of the proposed \ac{PDDP} algorithm and show it is globally convergent to a minimum of the cost for optimization problems with arbitrary dynamics and cost function parameters.
    \item Discuss and analyze three optimization choices for iteratively solving for the optimal controls and parameters.
    \item Apply the proposed method on multiple robotics systems, including a cartpole, a quadrotor, and an ant quadruped system using the open-source physics engine Brax \cite{brax2021github}, and show \ac{PDDP} is able to solve adaptive \ac{MPC} tasks using \ac{MHE}.
    \item Apply \ac{PDDP} to find the optimal transition point between multi-modal dynamical systems, including a \ac{UAM} class vehicle exhibiting multiple phases of flight.
\end{enumerate}

The paper is organized as follows. \cref{sec/pddp} derives the \ac{PDDP} algorithm and proves the convergence of the method. \cref{sec/applications} discusses applications of \ac{PDDP} to system parameter estimation and switching time optimization for hybrid systems. \cref{sec/experiments} analyzes experimental results applying \ac{PDDP} to the previously discussed tasks. The paper concludes with \cref{sec/conclusion} by discussing future work.

\section{Parameterized Differential Dynamic Programming}
\label{sec/pddp}

In this section, the parameterized version of \ac{DDP}
is derived.
While the similarities are highlighted here, a full derivation of standard \ac{DDP} can be found in recent works such as \cite{tassa2014control}.

\subsection{Problem Formulation}
This work considers parameterized discrete-time dynamics that evolve according to
\begin{align}
    \mathbf{x}_{t + 1} & = \mathbf{F}(\mathbf{x}_t, \mathbf{u}_t; \btheta) \label{eq/parameterized_dynamics}
,\end{align}
where $\mathbf{F}: \R^{n_x} \times \R^{n_u} \times \R^{n_\theta} \to \R^{n_x}$ with $\mathbf{x}_t \in \R^{n_x}$, $\mathbf{u}_t \in \R^{n_u}$, and $\btheta \in \R^{n_\theta}$, and $\mathbf{x}_1 \in \R^{n_x}$ denoting the initial condition. The semicolon emphasizes that the parameters are time-invariant and separate from states or controls that are time-varying.
The same assumptions made in standard \ac{DDP}~\cite{jacobson1970differential} are adopted here. Namely, the dynamics are assumed to be twice differentiable with respect to state and control, with the additional assumption that the dynamics are twice differentiable with respect to the parameters.

Defining the control sequence $\mathbf{U} \defeq \{\mathbf{u}_1, \ldots, \mathbf{u}_T\}$, the parameterized discrete-time optimal control problem is given by
\begin{equation} \label{eq/parameterized_cost_fn}
    \min_{\mathbf{U}, \btheta} \mathcal{J}(\mathbf{U}; \btheta) = \min_{\btheta} \min_{\mathbf{U}} \sum_{t = 1}^{T} \mathcal{L}(\mathbf{x}_t, \mathbf{u}_t; \btheta) + \phi(\mathbf{x}_{T + 1}; \btheta)
,\end{equation}
where $\mathcal{L}$ and $\phi$ are the twice differentiable running cost and terminal cost function, respectively, and $T$ is the time horizon.
In the most general case, the parameters $\btheta$ affect both the dynamics and the cost.

The inner objective of \cref{eq/parameterized_cost_fn} with $\btheta$ fixed is equivalent to the standard \ac{DDP} optimal control problem. The value function, now parameterized by $\btheta$, is given recursively by
\begin{align}
    V(\mathbf{x}_t; \btheta) & = \min_{\mathbf{u}_t} \Big[ \underbrace{\mathcal{L}(\mathbf{x}_t, \mathbf{u}_t; \btheta) + V(\mathbf{x}_{t + 1}; \btheta)}_{\defeq Q(\mathbf{x}_t, \mathbf{u}_t; \btheta)} \Big] \label{eq/parameterized_value_fn}
,\end{align}
with $V(\mathbf{x}_{T + 1}; \btheta) = \phi(\mathbf{x}_{T + 1}; \btheta)$. Thus, the optimal $\btheta$ is given at the initial time $t = 1$ by
\begin{align}
    \btheta^* & = \argmin_{\btheta} V(\mathbf{x}_1; \btheta) \label{eq/theta_star}
.\end{align}

\subsection{Algorithm Derivation}
\label{subsec/alg_derivation}
Given the parameterized dynamics defined by \cref{eq/parameterized_dynamics}, this subsection will derive an iterative algorithm for finding the control trajectory $\mathbf{U}$ and parameters $\btheta$ that minimize the cost function defined in \cref{eq/parameterized_cost_fn}.

As in standard \ac{DDP}, the parameterized optimal control problem is solved by considering quadratic approximations of the value function along a nominal trajectory $\bar{\mathbf{x}}_t$, $\bar{\mathbf{u}}_t$. However, this derivation also includes terms expanding the value function about a set of nominal parameters $\bar{\btheta}$. Let $\delta \mathbf{x}_t, \delta \mathbf{u}_t, \delta \btheta$ be the variation in state, control, and parameters, respectively, so that
\begin{align}
    \mathbf{x}_t \defeq \bar{\mathbf{x}}_t + \delta \mathbf{x}_t , && \mathbf{u}_t \defeq \bar{\mathbf{u}}_t + \delta \mathbf{u}_t , && \btheta \defeq \bar{\btheta} + \delta \btheta
.\end{align}

The quadratic expansion of the value function about $\bar{\mathbf{x}}_t, \bar{\btheta}$ has the form
\begin{align}
\setlength\arraycolsep{2pt}
    \begin{aligned}
        V(\mathbf{x}_t; \btheta) \approx {}& V^0_t + (V^x_t)^\top \delta \mathbf{x}_t + (V^\theta_t)^\top \delta \btheta \\
        & + \frac{1}{2} \begin{bmatrix}
            \delta \mathbf{x}_t \\ \delta \btheta
        \end{bmatrix}^\top \begin{bmatrix}
            V^{xx}_t & V^{x\theta}_t \\
            V^{\theta x}_t & V^{\theta\theta}_t
        \end{bmatrix} \begin{bmatrix}
            \delta \mathbf{x}_t \\ \delta \btheta
        \end{bmatrix}
    ,\end{aligned} \label{eq/parameterized_value_fn_approx}
\end{align}
where the partial derivatives of $V$ are denoted using superscripts and are evaluated at $\bar{\mathbf{x}}_t, \bar{\btheta}$, e.g., $V^0_t = V(\bar{\mathbf{x}}_t; \bar{\btheta})$, $V^x_t = \nabla_x V(\bar{\mathbf{x}}_t; \bar{\btheta})$, etc. This superscript notation will be adopted throughout the paper. In order to solve for the partial derivatives of $V$, the quadratic expansion of the $Q$ function defined in \cref{eq/parameterized_value_fn} is taken about $\bar{\mathbf{x}}_t$, $\bar{\mathbf{u}}_t, \bar{\btheta}$:
\begin{align}
\setlength\arraycolsep{2pt}
    \hspace{-1.0em}\begin{aligned}
        Q(\mathbf{x}_t, \mathbf{u}_t; \btheta) \approx {}& Q^0_t + (Q^x_t)^\top \delta \mathbf{x}_t + (Q^u_t)^\top \delta \mathbf{u}_t + (Q^\theta_t)^\top \delta \btheta \\
        & + \frac{1}{2} \begin{bmatrix}
            \delta \mathbf{x}_t \\ \delta \mathbf{u}_t \\ \delta \btheta
        \end{bmatrix}^\top \begin{bmatrix}
            Q^{xx}_t & Q^{xu}_t & Q^{x \theta}_t \\
            Q^{ux}_t & Q^{uu}_t & Q^{u \theta}_t \\
            Q^{\theta x}_t & Q^{\theta u}_t & Q^{\theta \theta}_t
        \end{bmatrix} \begin{bmatrix}
            \delta \mathbf{x}_t \\ \delta \mathbf{u}_t \\ \delta \btheta
        \end{bmatrix}
    ,\end{aligned} \label{eq/parameterized_Q_approximation}
\end{align}
where
\begin{equation} \label{eq/parameterized_Q_derivatives}
\begin{split}
    Q^0_t & = \mathcal{L}^0_t + V^0_{t + 1} ,\\
    Q^x_t & = \mathcal{L}^x_t + (\mathbf{F}^x_t)^\top V^x_{t + 1} ,\\
    Q^u_t & = \mathcal{L}^u_t + (\mathbf{F}^u_t)^\top V^x_{t + 1} ,\\
    Q^\theta_t & = \mathcal{L}^\theta_t + V^\theta_{t + 1} + (\mathbf{F}^\theta_t)^\top V^x_{t + 1} ,\\
    Q^{xx}_t & = \mathcal{L}^{xx}_t + (\mathbf{F}^x_t)^\top V^{xx}_{t + 1} \mathbf{F}^x_t ,\\
    Q^{xu}_t & = \mathcal{L}^{xu}_t + (\mathbf{F}^x_t)^\top V^{xx}_{t + 1} \mathbf{F}^u_t = (Q^{ux}_t)^\top ,\\
    Q^{x\theta}_t & = \mathcal{L}^{x\theta}_t + (\mathbf{F}^x_t)^\top V^{x\theta}_{t + 1} + (\mathbf{F}^x_t)^\top V^{xx}_{t + 1} \mathbf{F}^\theta_t = (Q^{\theta x}_t)^\top ,\\
    Q^{uu}_t & = \mathcal{L}^{uu}_t + (\mathbf{F}^u_t)^\top V^{xx}_{t + 1} \mathbf{F}^u_t ,\\
    Q^{u\theta}_t & = \mathcal{L}^{u\theta}_t + (\mathbf{F}^u_t)^\top V^{x\theta}_{t + 1} + (\mathbf{F}^u_t)^\top V^{xx}_{t + 1} \mathbf{F}^\theta_t = (Q^{\theta u}_t)^\top ,\\
    Q^{\theta\theta}_t & = \mathcal{L}^{\theta\theta}_t + V^{\theta\theta}_{t + 1} + 2 (\mathbf{F}^\theta_t)^\top V^{x\theta}_{t + 1} + (\mathbf{F}^\theta_t)^\top V^{xx}_{t + 1} \mathbf{F}^\theta_t
.\end{split}
\end{equation}
Note the partial derivatives with respect to $\btheta$ include extra terms because both the value function and the dynamics depend on $\btheta$. In \cref{eq/parameterized_Q_derivatives}, the second-order dynamics terms have been dropped following the \ac{iLQR} algorithm \cite{li2004iterative}. A full derivation with all second-order terms is given in \cref{proof/full_Q_derivatives}.

Substituting this quadratic approximation into \cref{eq/parameterized_value_fn} and dropping the terms not dependent on $\delta \mathbf{u}_t$ gives
\begin{align}
    \begin{aligned}
    V(\mathbf{x}_t; \btheta) = \min_{\delta \mathbf{u}_t} \Big[ & (Q^u_t)^\top \delta \mathbf{u}_t + \delta \mathbf{x}_t^\top Q^{xu}_t \delta \mathbf{u}_t \\
    & + \delta \btheta^\top Q^{\theta u}_t \delta \mathbf{u}_t + \frac{1}{2} \delta \mathbf{u}_t^\top Q^{uu}_t \delta \mathbf{u}_t \Big]
    .\end{aligned}
\end{align}
Taking the gradient of the argument in the minimization and setting it equal to zero yields the optimal control update
\begin{align} \label{eq/control_update_unscaled}
    \delta \mathbf{u}_t^* & = \mathbf{k}_t + \mathbf{K}_t \delta \mathbf{x}_t + \mathbf{M}_t \delta \btheta
\end{align}
with
\begin{equation} \label{eq/parameterized_gains}
    \begin{aligned}
        \mathbf{k}_t & \defeq -(Q^{uu}_t)^{-1} Q^u_t ,\\
        \mathbf{K}_t & \defeq -(Q^{uu}_t)^{-1} Q^{ux}_t ,\\
        \mathbf{M}_t & \defeq -(Q^{uu}_t)^{-1} Q^{u\theta}_t .
    \end{aligned}
\end{equation}

Note the gains $\mathbf{k}_t$ and $\mathbf{K}_t$ match the standard \ac{DDP} gains, with an additional feedback gain $\mathbf{M}_t$ on $\delta \btheta$ providing a correction because the algorithm is optimizing over $\delta \mathbf{u}_t$ and $\delta \btheta$ simultaneously. Taking $\delta \btheta = \mathbf{0}$ in \cref{eq/control_update_unscaled} means the parameters are held constant which yields the usual \ac{DDP} control update.

The optimal $\btheta$ minimizes \cref{eq/parameterized_value_fn} at time $t = 1$. Since the initial condition is fixed, $\mathbf{x}_1 = \bar{\mathbf{x}}_1$, which implies $\delta \mathbf{x}_1 = \mathbf{0}$. Substituting in the quadratic approximation of $Q$ into \cref{eq/theta_star} and dropping the terms not dependent on $\delta \btheta$, yields
\begin{align}
    \delta \btheta^* & = \argmin_{\delta \btheta} \Big[ (Q^\theta_1)^\top \delta \btheta + \delta \mathbf{u}_1^\top Q^{u\theta}_1 \delta\btheta + \frac{1}{2} \delta\btheta^\top Q^{\theta\theta}_1 \delta\btheta \Big] \nonumber \\
    & = - (Q^{\theta\theta}_1)^{-1} Q^\theta_1 - (Q^{\theta\theta}_1)^{-1} Q^{\theta u}_1 \delta \mathbf{u}_1
.\end{align}
At the initial time $t = 1$, there are two equations and two unknowns $\delta \btheta^*$ and $\delta \mathbf{u}_1^*$. Substituting in $\delta \mathbf{u}_1^*$ from \cref{eq/control_update_unscaled} results in the parameter update
\begin{align}
    \begin{aligned}
        \delta \btheta^* = \mathbf{m} \defeq -(Q^{\theta\theta}_1 - Q^{\theta u}_1 (Q^{uu}_1)^{-1} Q^{u\theta}_1)^{-1} \\
        (Q^{\theta}_1 - Q^{\theta u}_1 (Q^{uu}_1)^{-1} Q^u_1)
    .\end{aligned} \label{eq/parameter_update_gain}
\end{align}
When the feedforward gains $\mathbf{k}_t$ of \cref{eq/control_update_unscaled} and $\mathbf{m}$ of \cref{eq/parameter_update_gain} are too aggressive, the state trajectory may stray from the region where the quadratic approximation is accurate, and the cost may not decrease. To ensure convergence, the feedforward gains are scaled by the parameter $0 < \epsilon \leq 1$, such that
\begin{align}
    \delta \mathbf{u}_t^* & = \epsilon \mathbf{k}_t + \mathbf{K}_t \delta \mathbf{x}_t + \mathbf{M}_t \delta \btheta^* , \label{eq/optimal_control_update} \\
    \delta \btheta^* & = \epsilon \mathbf{m} . \label{eq/scaled_optimal_parameter_update}
\end{align}
The parameter $\epsilon$ is determined using line search and ensures the optimizer makes sufficient progress towards the minimum at each iteration, i.e. ensuring the update step satisfies Armijo's condition or the Wolfe conditions \cite{boyd2004convex}. Further discussion is given in \cref{subsec/convergence}.

What remains is to solve the value function derivatives by substituting in the expression for $\delta \mathbf{u}_t^*$. The expression for $\delta\btheta^*$ is not substituted back in, otherwise the terms in $\delta \btheta^*$ are incorporated into the zero-order value function term and the derivatives with respect to $\btheta$ go to zero. This means that the information for $\btheta$ cannot be propagated backwards in time.
In \cref{subsec/convergence}, to prove the cost reduction achievable, the full expression for $\delta\btheta^*$ is substituted into \cref{eq/parameterized_Q_approximation}.

Substituting \cref{eq/optimal_control_update} into \cref{eq/parameterized_Q_approximation} and equating like powers of the value function expansion in \cref{eq/parameterized_value_fn_approx} gives the following expressions for the value function derivatives:
\begin{equation} \label{eq/parameterized_value_fn_derivatives}
    \begin{split}
        V^0_t & = Q^0_t + (\frac{1}{2} \epsilon^2 - \epsilon) (Q^u_t)^\top (Q^{uu}_t)^{-1} Q^u_t ,\\
        V^x_t & = Q^x_t - Q^{xu}_t (Q^{uu}_t)^{-1} Q^u_t ,\\
        V^\theta_t & = Q^\theta_t - Q^{\theta u}_t (Q^{uu}_t)^{-1} Q^u_t ,\\
        V^{xx}_t & = Q^{xx}_t - Q^{xu}_t (Q^{uu}_t)^{-1} Q^{ux}_t ,\\
        V^{x\theta}_t & = Q^{x\theta}_t - Q^{xu}_t (Q^{uu}_t)^{-1} Q^{u\theta}_t = (V^{\theta x}_t)^\top ,\\
        V^{\theta\theta}_t & = Q^{\theta\theta}_t - Q^{\theta u}_t (Q^{uu}_t)^{-1} Q^{u \theta}_t .
    \end{split}
\end{equation}

\cref{eq/parameterized_value_fn_derivatives} and \cref{eq/parameterized_Q_derivatives} provide the equations for the backward pass of \ac{PDDP}, with boundary conditions $V^0_{T + 1} = \phi(\bar{\mathbf{x}}_{T + 1}; \bar{\btheta})$, $V^x_{T + 1} = \nabla_x \phi(\bar{\mathbf{x}}_{T + 1}; \bar{\btheta})$, etc. The derivatives of $V$ and $Q$ are solved for backwards in time starting from $t = T + 1$ down to the initial time $t = 1$ along the nominal trajectory. The forward pass of the algorithm consists of using the control updates from \cref{eq/optimal_control_update} and parameter updates from \cref{eq/scaled_optimal_parameter_update} to compute the new state trajectory starting from the initial condition $\mathbf{x}_1$. This yields an updated nominal trajectory, and this process can be repeated until some convergence criteria is met. Discussion on regularization of the value function derivatives as well as methods for updating the controls and parameters is given in \cref{subsec/alg_design} once the convergence of the method is proven.

\subsection{Convergence Analysis}
\label{subsec/convergence}

This section provides a mathematically rigorous convergence analysis of the proposed \ac{PDDP} algorithm. The main result is summarized in \cref{thm/pddp_converges}, which shows \ac{PDDP} is globally convergent to a minimum of the cost function. The same assumptions as in \cite{jacobson1970differential} for classic \ac{DDP} are made in this work, including the differentiability of the dynamics and cost function up to second-order.

The following proposition establishes the optimality of the parameter update step by drawing comparisons to Newton's method.

\begin{proposition} \label{prop/parameters_are_newton_step}
    The parameter update $\delta \btheta^*$ is a (damped) Newton step towards the minimum of the value function.
\end{proposition}
\begin{proof}
    See \cref{proof/parameters_are_newton_step}.
\end{proof}

\cref{prop/parameters_are_newton_step} implies \ac{PDDP} converges quadratically fast to the optimal parameters $\btheta^*$, adopting the convergence rate of Newton's method. It is important to note that \ac{PDDP} differs from a pure stagewise Newton's method in that it uses the full nonlinear dynamics during the forward pass, thus achieving better numerical convergence properties \cite{liao1992advantages}.

However, for general nonlinear, nonconvex problems, it cannot be guaranteed that $V^{\theta\theta}_1$ or $Q^{uu}_1$ remain positive definite, which motivates the addition of regularization to ensure convergence. Regularization can be accomplished through addition of a Levenberg-Marquardt parameter ensuring the Hessian matrices are always positive definite \cite{todorov2005generalized}. Throughout this work, the regularization scheme proposed in \cite{tassa2012synthesis} is adopted. Further discussion is given in \cref{subsec/alg_design}.

Next, the control updates are proven to reduce the cost after each iteration of \ac{PDDP}. To this end, an expression for the gradient of the cost function with respect to an individual control is derived, and it is shown that the variations in control, state, and parameters are $O(\epsilon)$.

\begin{lemma} \label{lemma/cost_fn_gradient}
    The gradient of the cost function with respect to the control at time $t$ is given by
    \begin{align}
        \nabla_{\mathbf{u}_t} \mathcal{J} & = \mathcal{L}^u_t + (\mathbf{F}^u_t)^\top \eta_{t + 1} \label{eq/J_u}
    ,\end{align}
    with $\eta_t = \mathcal{L}^x_t + (\mathbf{F}^x_t)^\top \eta_{t + 1}$ for $t = 1, \ldots, T$ and $\eta_{T + 1} = \phi^x_{T + 1}$.
\end{lemma}
\begin{proof}
    See \cref{proof/cost_fn_gradient}.
\end{proof}

\begin{lemma} \label{lemma/updates_order_epsilon}
    For the form of $\delta \btheta$ given in \cref{eq/scaled_optimal_parameter_update}, it is true that
    \begin{align}
        \delta \btheta & = O(\epsilon)
    ,\end{align}
    where $O(\cdot)$ corresponds to big-O notation in $\norm{\cdot}_2$.

    Further, for all $t = 1, \ldots, T$,
    \begin{subequations} \label{eq/delta_u_x_O_epsilon}
        \begin{align}
            \delta \mathbf{u}_t & = O(\epsilon) ,\\
            \delta \mathbf{x}_{t + 1} & = O(\epsilon) .
        \end{align}
    \end{subequations}
\end{lemma}
\begin{proof}
    See \cref{proof/updates_order_epsilon}.
\end{proof}

Next, the control updates are shown to be a descent direction of the cost.
\begin{proposition} \label{prop/control_updates_are_descent_direction}
    The optimal control updates satisfy
    \begin{align}
        (\nabla_{\mathbf{U}} \mathcal{J})^\top \Delta \mathbf{U} & = \sum_{t = 1}^{T} (\nabla_{\mathbf{u}_t} \mathcal{J})^\top \delta \mathbf{u}_t < 0 \label{eq/control_updates_descent_direction}
    .\end{align}
\end{proposition}
\begin{proof}
    It can be shown that
    \begin{equation*}
        \begin{split}
            \sum_{t = 1}^{T} & (\nabla_{\mathbf{u}_t} \mathcal{J})^\top \delta \mathbf{u}_t \\
            & = - \epsilon \sum_{t = 1}^{T} \lambda_t + \epsilon \left( \sum_{t = 1}^{T} \gamma_t \right) (V^{\theta\theta}_1)^{-1} V^\theta_1 + O(\epsilon^2)
        ,\end{split}
    \end{equation*}
    where $\lambda_t = (Q^u_t)^\top (Q^{uu}_t)^{-1} Q^u_t$ and $\gamma_t = (Q^u_t)^\top (Q^{uu}_t)^{-1} Q^{u \theta}_t$. This implies there exists some $\epsilon$ sufficiently small such that \cref{eq/control_updates_descent_direction} holds. The full proof is given in \cref{proof/control_updates_are_descent_direction}.
\end{proof}
\begin{remark}
    The row vector $\sum_{t = 1}^T \gamma_t$ can be thought of as correcting for the change in $\btheta$, as both the parameters and the controls are optimized simultaneously. The inner product between $\sum_{t = 1}^T \gamma_t$ and the parameter ascent direction $(V^{\theta\theta}_1)^{-1} V^{\theta}_1$ ensures that the optimizer does not take too large of a step in the case when the controls and the parameters both contribute to a reduction in cost.
\end{remark}

To show the cost reduction achievable after a single iteration of \ac{PDDP}, the expression for $\delta\btheta^*$ must be substituted into the value function approximation given by \cref{eq/parameterized_Q_approximation}. This substitution means the terms in $\delta\btheta$ are now incorporated into the zero-order term $V^0_t$. The previous expression for $V^0_t$ given in \cref{eq/parameterized_value_fn_derivatives} only accounts for the change in controls $\delta \mathbf{u}_t^*$, which was necessary to derive a backwards rule for the value function derivatives with respect to $\btheta$. The following lemma establishes the value function change by applying the updates $\delta \mathbf{u}_t^*$ and $\delta \btheta^*$.

\begin{lemma} \label{lemma/value_fn_full_zero_order_term}
    The zero-order value function approximation term satisfies
    \begin{align}
        V^0_t & = Q^0_t - \epsilon (1 - \frac{1}{2} \epsilon) \lambda_t + \epsilon (V^\theta_t)^\top \mathbf{m} + \frac{1}{2} \epsilon^2 \mathbf{m}^\top V^{\theta\theta}_t \mathbf{m} \label{eq/V_0_t_full}
    .\end{align}

    Further, at time $t = 1$, the following is true:
    \begin{align}
        V^0_1 & = Q^0_1 - \epsilon (1 - \frac{1}{2} \epsilon) (\lambda_1 + \psi) \label{eq/V_0_1}
    ,\end{align}
    where $\psi = (V^\theta_1)^\top (V^{\theta\theta}_1)^{-1} V^\theta_1$.
\end{lemma}
\begin{proof}
    See \cref{proof/value_fn_full_zero_order_term}.
\end{proof}
\begin{remark}\label{remark/newton_decrement}
    $\psi$ can be interpreted as the Newton decrement on the parameters. Likewise, $\lambda_t$ can be interpreted as the Newton decrement on the controls at each timestep. These values can be used as an effective stopping criteria for the algorithm as well as verification for the line search parameter $\epsilon$ \cite{boyd2004convex}. Further discussion is provided in the following proposition.
\end{remark}

Now, it is possible to show the cost reduction achievable after an iteration of \ac{PDDP}.
\begin{proposition}\label{prop/cost_reduction}
    The total cost reduction after each iteration of \ac{PDDP} is
    \begin{align}
        \Delta \mathcal{J} & = -\epsilon (1 - \frac{1}{2} \epsilon) \left(\sum_{t = 1}^T \lambda_t + \psi \right) + O(\epsilon^3)
    .\end{align}
\end{proposition}
\begin{proof}
    See \cref{proof/cost_reduction}.
\end{proof}
\begin{remark}
    \cref{prop/cost_reduction} implies that a valid line search candidate at each iteration should satisfy the following cost reduction:
    \begin{align}
        \Delta \mathcal{J} \leq - \kappa \epsilon \left( \sum_{t = 1}^T \lambda_t + \psi \right)
    ,\end{align}
    for a small constant $\kappa > 0$. This condition is similar to the line search criteria for Newton's method, and ensures that Armijo's condition or the Wolfe conditions hold \cite{boyd2004convex}.
\end{remark}

Finally, the convergence of \ac{PDDP} to a minimum of the cost is proven.
\begin{theorem} \label{thm/pddp_converges}
    As the number of iterations of \ac{PDDP} approaches infinity, the cost $\mathcal{J}$, the control trajectory $\mathbf{U}$, and the parameters $\btheta$ converge to a stationary point regardless of initialization.
\end{theorem}
\begin{proof}
    See \cref{proof/pddp_converges}.
\end{proof}

\subsection{Algorithm Design}
\label{subsec/alg_design}

The convergence result presented in \cref{subsec/convergence} assumes that the Hessian matrices $Q^{uu}_t$ and $V^{\theta\theta}_1$ remain positive definite throughout the optimization process. One method of ensuring this condition always holds is to use Levenberg-Marquardt regularization corresponding to adding a quadratic cost around the nominal trajectory \cite{todorov2005generalized}. This regularization can be employed to ensure $V^{\theta\theta}_1$ is positive definite since this Hessian matrix only needs to be invertible at the first timestep. However, the authors of \cite{tassa2012synthesis} have shown a more robust approach for regularizing $Q^{uu}_t$ places the regularization on the states rather than the controls, ensuring the control perturbation does not have different effects at different timesteps. Therefore, the regularized derivatives can be computed using
\begin{equation} \label{eq/regularized_derivatives}
    \begin{split}
		\tilde{V}^{xx}_{t + 1} & = V^{xx}_{t + 1} + \mu \mathbf{I} \\
        \tilde{V}^{\theta\theta}_1 & = V^{\theta\theta}_1 + \nu \mathbf{I}
    ,\end{split}
\end{equation}
with hyperparameters $\mu, \nu \geq 0$. The regularized Hessian $\tilde{V}^{xx}_{t + 1}$ is used in place of the original during the backward pass calculation of \cref{eq/parameterized_Q_derivatives}. The regularization of $V^{xx}_{t + 1}$ and $V^{\theta\theta}_1$ correspond to adding a quadratic cost to the deviation of the state and parameters from the nominal values, thus ensuring the update step does not stray too far from the region where the quadratic approximation is accurate. However, the addition of the regularization on $V^{xx}_{t + 1}$ when calculating the derivatives in \cref{eq/parameterized_value_fn_derivatives} means the same cancellations of $Q^{uu}_t$ and its inverse are not possible. The updated derivatives are given as
\begin{equation} \label{eq/full_value_fn_derivatives}
	\begin{split}
		V^x_t & = Q^{x}_t + \mathbf{K}_t^\top Q^{u}_t + Q^{xu}_t \mathbf{k}_t + \mathbf{K}_t^\top Q^{uu}_t  \mathbf{k}_t, \\
		V^\theta_t & = Q^\theta_t +  \mathbf{M}^\top_t Q^u_t + Q^{\theta u}_t  \mathbf{k}_t +  \mathbf{M}^\top_t Q^{uu}_t  \mathbf{k}_t, \\
		V^{xx}_t & = Q^{xx}_t + Q^{xu}_t  \mathbf{K}_t +  \mathbf{K}_t^\top Q^{ux}_t +  \mathbf{K}_t^\top Q^{uu}_t  \mathbf{K}_t, \\
		V^{x\theta}_t & = Q^{x\theta}_t + Q^{xu}_t  \mathbf{M}_t +  \mathbf{K}_t^\top Q^{u\theta}_t +  \mathbf{K}_t^\top Q^{uu}_t  \mathbf{M}_t \\
		V^{\theta\theta}_t & = Q^{\theta\theta}_t + Q^{\theta u }_t \mathbf{M}_t + \mathbf{M}^\top_t Q^{u \theta}_t + \mathbf{M}^\top_t Q^{uu}_t \mathbf{M}_t.
	\end{split}
\end{equation}

An additional benefit of the proposed method is that the optimization of the controls and parameters can be coupled or decoupled without affecting the convergence. The result in \cref{thm/pddp_converges} assumes both the control and parameter feedforward gains are applied simultaneously, but the same result holds if the gains are applied in an alternating fashion, e.g. at every odd iteration the controls are updated by applying the control feedforward gain $\mathbf{k}_t$ and at every even iteration the parameters are updated by applying the parameter feedforward gain $\mathbf{m}$. The addition of the feedback term $\mathbf{M}_t \delta \btheta$ on the parameter update when calculating the updated controls in \cref{eq/optimal_control_update} ensures stability in the controls even when only the parameters are being updated. Likewise, \cite{li2020hybrid} adopt an approach where the parameters are only updated once the controls have fully converged. This approach is referred to as \ac{STO-DDP}. The alternating and simultaneous schemes are benchmarked against \ac{STO-DDP} through a numerical comparison on three systems, which is presented in \cref{subsec/sto_experiments}.

In summary, \cref{alg/pddp} describes the proposed \ac{PDDP} algorithm.

\begin{algorithm}
    \caption{Parameterized Differential Dynamic Programming}\label{alg/pddp}
    \KwIn{Nominal trajectory $\bar{\mathbf{x}}_t, \bar{\mathbf{u}}_t$ and parameters $\bar{\btheta}$}
    \While{not converged}{
        Calculate derivatives of $\cal{L}$, $\phi$ along $\bar{\mathbf{x}}_t, \bar{\mathbf{u}}_t, \bar{\btheta}$ \\
        \tcp{Backward pass}
        $V^0_{T + 1} \gets \phi^0_{T + 1}$, $V^x_{T + 1} \gets \phi^x_{T + 1}$, $V^\theta_{T + 1} \gets \phi^\theta_{T + 1}$, \\
        $V^{xx}_{T + 1} \gets \phi^{xx}_{T + 1}$, $V^{x\theta}_{T + 1} \gets \phi^{x\theta}_{T + 1}$, $V^{\theta\theta}_{T + 1} \gets \phi^{\theta\theta}_{T + 1}$ \\
        \For{$t = T, \ldots, 1$}{
            Calculate derivatives of $Q$ according to \cref{eq/parameterized_Q_derivatives} and \cref{eq/regularized_derivatives} \\
            Calculate gains $\mathbf{k}_t, \mathbf{K}_t, \mathbf{M}_t$ according to \cref{eq/parameterized_gains} \\
            Calculate derivatives of $V$ according to \cref{eq/full_value_fn_derivatives}
        }
        \tcp{End backward pass}
        Calculate gain $\mathbf{m}$ according to \cref{eq/parameter_update_gain} \\
        \tcp{Line search}
        $\epsilon \gets 1$ \\
        \While{cost decrease not sufficient}{
            \tcp{Forward pass}
            \uIf{update parameters}{
                $\delta \btheta \gets \epsilon \mathbf{m}$ \\
            }
            \Else{
                $\delta \btheta \gets \mathbf{0}$
            }
            $\btheta \gets \bar{\btheta} + \delta \btheta$ \\
            $\mathbf{x}_1 \gets \bar{\mathbf{x}}_1$ \\
            \For{$t = 1, \ldots, T$}{
                $\delta \mathbf{x}_t \gets \mathbf{x}_t - \bar{\mathbf{x}}_t$ \\
                \uIf{update controls}{
                    $\delta \mathbf{u}_t \gets \epsilon \mathbf{k}_t + \mathbf{K}_t \delta \mathbf{x}_t + \mathbf{M}_t \delta \btheta$
                }
                \Else{
                    $\delta \mathbf{u}_t \gets \mathbf{K}_t \delta \mathbf{x}_t + \mathbf{M}_t \delta \btheta$
                }
                $\mathbf{u}_t \gets \bar{\mathbf{u}}_t + \delta \mathbf{u}_t$ \\
                $\mathbf{x}_{t + 1} \gets \mathbf{F}(\mathbf{x}_t, \mathbf{u}_t; \btheta)$
            }
            \tcp{End forward pass}
            $\epsilon \gets \rho \epsilon$ \quad \tcp{Reduce $\epsilon$}
        }
        \tcp{End line search}
        $\bar{\mathbf{x}}_t \gets \mathbf{x}_t$, $\bar{\mathbf{u}}_t \gets \mathbf{u}_t$, $\bar{\btheta} \gets \btheta$
    }
    \Return{optimal controls $\mathbf{u}_t$, optimal parameters $\btheta$}
\end{algorithm}

\section{Applications}
\label{sec/applications}

In this section the proposed \ac{PDDP} algorithm is applied to two important robotics control tasks: parameter estimation and switching time optimization.

\subsection{Parameter Estimation}
\label{subsec/parameter_est}
Estimating the unknown parameters and states of a dynamical system can be achieved through \acf{MHE}. This framework reformulates the state estimation problem as an optimization problem dual to \ac{MPC} \cite{Diehl2009}. This work assumes full state observability, with the realized discrete-time dynamics given by
\begin{equation*}
\begin{split}
    \mathbf{x}_{t + 1} & = \mathbf{F}(\mathbf{x}_t, \mathbf{u}_t; \btheta^*) + \mathbf{w}_t , \\
    \mathbf{x}_1 & \sim \mathcal{N}(\hat{\mathbf{x}}_1, \boldsymbol{\Sigma}_{\mathbf{x}_1}) ,
\end{split}
\end{equation*}
where $\mathbf{w}_t \sim \mathcal{N}(\mathbf{0}, \boldsymbol{\Sigma}_w)$ is some unknown additive Gaussian process noise. If the controls are given, this corresponds to a trajectory smoothing problem where the goal is to identify the parameters $\btheta$ and the realized disturbances $\mathbf{w}_t$ that maximize the likelihood of the observed states. The objective can be formulated through \acf{MLE} by minimizing the negative log-likelihood
\begin{align*}
    \mathcal{J}_{\text{est}} (\btheta, \mathbf{x}_1) & = - \log p(\mathbf{x}_1, \btheta) \prod_{t = 1}^{T} p(\mathbf{x}_{t + 1} | \mathbf{x}_{t}, \btheta) \\
    & = - \log p(\mathbf{x}_1, \btheta) - \sum_{t = 1}^{T} \log p(\mathbf{x}_{t + 1} | \mathbf{x}_{t}, \btheta)
,\end{align*}
with $p(\mathbf{x}_{t + 1} | \mathbf{x}_{t}, \btheta)$ describing the likelihood of observing the next state $\mathbf{x}_{t + 1}$ given the current state $\mathbf{x}_{t}$ and parameters $\btheta$, and $p(\mathbf{x}_1, \btheta)$ describing the prior distribution over the initial state and parameters.

Assuming a Gaussian prior over the parameters $\btheta \sim \mathcal{N}(\hat{\btheta}, \boldsymbol{\Sigma}_\theta)$, the estimation cost takes the form
\begin{align}
    \begin{aligned}
        \mathcal{J}_{\text{est}} (\btheta, \mathbf{x}_1) & = \frac{1}{2}  \sum_{t = 1}^{T_\text{est}} \norm{\mathbf{x}_{t + 1} - \mathbf{F}(\mathbf{x}_t, \mathbf{u}_t; \btheta)}_{\boldsymbol{\Sigma}_{w}^{-1}}^2 \\
        & \qquad + \frac{1}{2} \lVert \btheta - \hat{\btheta} \rVert_{\boldsymbol{\Sigma}_{\theta}^{-1}}^2 + \frac{1}{2} \norm{\mathbf{x}_1 - \hat{\mathbf{x}}_1}_{\boldsymbol{\Sigma}_{\mathbf{x}_1}^{-1}}^2 .
    \end{aligned} \label{eq/estimation_cost}
\end{align}

The terms $\lVert \btheta - \hat{\btheta} \rVert_{\boldsymbol{\Sigma}_{\theta}^{-1}}^2$ and $\norm{\mathbf{x}_1 - \hat{\mathbf{x}}_1}_{\boldsymbol{\Sigma}_{\mathbf{x}_1}^{-1}}^2$ ensure the new estimates of the parameters and initial state do not deviate too far from the prior, while $\norm{\mathbf{x}_{t + 1} - \mathbf{F}(\mathbf{x}_t, \mathbf{u}_t; \btheta)}_{\boldsymbol{\Sigma}_{w}^{-1}}^2$ ensures the predicted state trajectory matches the actual observed states. This cost function setup allows \ac{PDDP} to be used to identify the true system parameters as well as an updated estimate of the true state. The observed state trajectory is assumed to be given, and the parameters are updated to match the model predictions to the observed states through minimization of \cref{eq/estimation_cost}.

The time horizon $T_\text{est}$ corresponds to the estimation horizon, i.e. the number of steps in the history to estimate. In practice, the horizon is shifted forward every step when the number of observed states exceeds the time horizon $T_\text{est}$ to maintain computational tractability, particularly in realtime applications. This approach is why \ac{MHE} is often referred to as the dual problem to \ac{MPC} \cite{Diehl2009}, as the same receding horizon approach is adopted to solve realtime tasks using solvers such as \ac{DDP}~\cite{tassa2012synthesis}. In the \ac{MPC} problem, the goal is to find the sequence of controls that minimizes a finite-horizon cost function of the form
\begin{align}
    \mathcal{J}_{\text{mpc}} (\mathbf{U}) & = \sum_{t = 1}^{T_{\text{mpc}}} \mathcal{L}(\mathbf{x}_t, \mathbf{u}_t) + \phi(\mathbf{x}_{T_{\text{mpc}} + 1})
.\end{align}

While \ac{MHE} optimizes over the past horizon $T_\text{est}$ up to the current state at time $t$, \ac{MPC} plans over the future time horizon from the current time $t$ up to $T_\text{mpc}$.
In both methods, as a new state observation is made, the time horizon is shifted forward one step and the optimization proceeds again.
A combined cost can be derived to optimize over both tasks simultaneously using \ac{PDDP}:
\begin{align}
    \mathcal{J}_{\text{combined}}(\mathbf{U}; \btheta, \mathbf{x}_1) & = \mathcal{J}_{\text{est}} (\btheta, \mathbf{x}_1) + \mathcal{J}_{\text{mpc}} (\mathbf{U})
.\end{align}
This combined cost function is used in \cref{subsec/est_experiments} to find the optimal parameters of a system while solving an \acs{MPC} task.
At each timestep $t$, \acs{PDDP} minimizes the \acs{MHE} cost over the past time horizon $[t - T_\text{est}, t]$, updating the parameter estimate to best match the observed states. Simultaneously, \acs{PDDP} minimizes the \acs{MPC} cost function, planning over the future time horizon $[t, t + T_\text{mpc}]$ and generating a control update corresponding to standard \acs{DDP} plus feedback based on the simultaneous update of the parameters. This approach can be interpreted as a variation of adaptive \ac{MPC} \cite{adetola2009adaptive} since the system parameters are adapted to fit the true model at every step of the \ac{MPC} algorithm. The simultaneous optimization is possible due to the explicit feedback gains provided during the derivation of \ac{PDDP}. This generalizes previous work by \cite{kobilarov2015differential}, which uses \ac{DDP} to only solve the \ac{MHE} task.

\subsection{Switching Time Optimization}
\label{subsec/sto}
Switching time optimization (STO) is proposed by the authors of \cite{li2020hybrid} and reformulates the trajectory optimization of hybrid systems into an equivalent objective that is easier to solve using iterative numerical solvers, including \ac{DDP}.

A hybrid system is defined by a sequence of $N$ modes such that for each mode $i = 1, \ldots, N$, the dynamics obey
\begin{align}
    \dot{\mathbf{x}} & = \mathbf{f}_i(\mathbf{x}(t), \mathbf{u}(t)) \label{eq/multimodal_dynamics}
.\end{align}
Example hybrid systems include bipedal robotic animals, which transition between modes of flight and contact with the ground while running, and \ac{UAM} class vehicles, which exhibit modes of vertical takeoff and landing, fixed-wing cruise, and a transition phase between the two.

Given this fixed sequence of modes, the goal in hybrid systems trajectory optimization is to find the optimal control trajectory and the optimal set of times $\mathbf{t} = [t_1, \ldots, t_N]^\top$ that minimize the cost, with each $t_i$ describing the terminal time of mode $i$. This problem can be formulated through the continuous time objective
\begin{align}
    \min_{\mathbf{u}, \mathbf{t}} \sum_{i = 1}^N \left[ \int_{t_{i - 1}}^{t_i} \ell_i(\mathbf{x}(t), \mathbf{u}(t)) \dd{t} + \varphi_i(\mathbf{x}(t_i)) \right] \label{eq/continuous_hybrid_ocp}
,\end{align}
with the dynamics transitioning accordingly through the modes defined in \cref{eq/multimodal_dynamics}.

The authors of \cite{li2020hybrid} show that this objective can be equivalently expressed as trying to find the optimal amount of time to spend in each mode through \ac{STO}. The dynamics in \cref{eq/multimodal_dynamics} and objective in \cref{eq/continuous_hybrid_ocp} are reformulated over fixed time intervals of unit length, with the dynamics and running cost scaled by the amount of time that is spent in each mode. Letting $\theta_i = t_i - t_{i - 1}$ and adopting the change of variable $\tau = (t - t_{i - 1})/\theta_i + i - 1$ results in the time scaled dynamics
\begin{align}
    \dot{\mathbf{x}} & = \theta_i \mathbf{f}_i(\mathbf{x}(\tau), \mathbf{u}(\tau)) \label{eq/time_scaled_dynamics}
,\end{align}
and the objective can be rewritten as
\begin{align}
    \min_{\mathbf{u}, \btheta} \sum_{i = 1}^N \left[ \int_{i - 1}^{i} \theta_i \ell_i(\mathbf{x}(\tau), \mathbf{u}(\tau)) \dd{\tau} + \varphi_i(\mathbf{x}(i)) \right] \label{eq/time_scaled_ocp}
.\end{align}

The equivalent discrete-time optimal control problem is given as
\begin{subequations} \label{eq/discrete_sto_ocp}
    \begin{align}
        & \min_{\mathbf{U}, \mathbf{t}} \sum_{i = 1}^N \left[ \phi_i(\mathbf{x}_{T_i + 1}) + \sum_{t = T_{i - 1} + 1}^{T_i} \theta_i \mathcal{L}_i(\mathbf{x}_t, \mathbf{u}_t) \right] \\
        & \text{subject to } \mathbf{x}_{t + 1} = \mathbf{F}_i(\mathbf{x}_t, \mathbf{u}_t; \theta_i) \label{subeq/discrete_time_scaled_dynamics}
    ,\end{align}
\end{subequations}
with a chosen integration step size $h$ and $T_i = T_{i - 1} + \theta_i / h$, $T_0 = 0$ denoting the number of timesteps in the horizon up to the end of mode $i$. This objective is a parameterized discrete-time optimal control problem that can be solved using \ac{PDDP}. In this problem, the optimized parameters are the intervals of time spent in each mode $\theta_i$. For a single mode, this objective is an approximation to a free-horizon optimal control problem, where the terminal time of the trajectory parameterizes the dynamics and can be optimized using \ac{PDDP}. Similar parameterizations have been discussed in previous works such as \cite{liu2021second, stachowicz2021optimal}.

Three key improvements are introduced to stabilize the optimization using a numerical solver such as \ac{PDDP}. First, the number of timesteps per mode is set to be constant, meaning timesteps do not have to be interpolated when the amount of time in a mode changes during the optimization. This improvement ensures that the convergence properties described in \cref{subsec/convergence} hold, as changing the number of timesteps in the optimization modifies the objective, which can cause unexpected problems and results in poor convergence.

Second, an improved integration scheme is adopted for calculating the discrete-time dynamics. Using the typical forward Euler integration with fixed step size $h$, the discrete-time dynamics from \cref{subeq/discrete_time_scaled_dynamics} are given as
\begin{align}
    \mathbf{x}_{t + 1} = \mathbf{x}_t + \theta_i \mathbf{f}_i(\mathbf{x}_t, \mathbf{u}_t) h
.\end{align}
Note the time scaling in the dynamics by the amount of time spent in each mode $\theta_i$ results in an equivalent scaling of the integration step size $h$ used in the forward Euler integration, which causes inaccuracies in the dynamics, particularly as the time scale $\theta_i$ becomes large. To remedy this issue, the dynamics are integrated at a smaller underlying step size $\Delta t$ for multiple substeps, which prevents numerical errors in the dynamics causing the optimization to fail.

Finally, constraints are necessary in order to ensure the time scales are always nonnegative. A popular choice to handle box control limits is the projected Newton quadratic program solver proposed in \cite{tassa2014control}. The same projected Newton solver is used when calculating the parameter feedforward gain $\mathbf{m}$ to ensure that the time scales never become negative.

\section{Experiments}
\label{sec/experiments}

In this section, \ac{PDDP} is applied to two separate tasks: optimal parameter estimation and switching time optimization. The results show \ac{PDDP} can solve \ac{MHE} and \ac{MPC} problems simultaneously, as well as optimize for time-optimal trajectories through \ac{STO}.

\subsection{Parameter Estimation}
\label{subsec/est_experiments}

The dual optimization over the combined \ac{MHE} and \ac{MPC} tasks proposed in \cref{subsec/parameter_est} is solved by \ac{PDDP} on multiple nonlinear, underactuated systems including a cartpole, quadrotor, and ant quadruped.

\subsubsection{Cartpole}
\begin{figure}[h]
    \centering
    \includegraphics[width=0.25\textwidth]{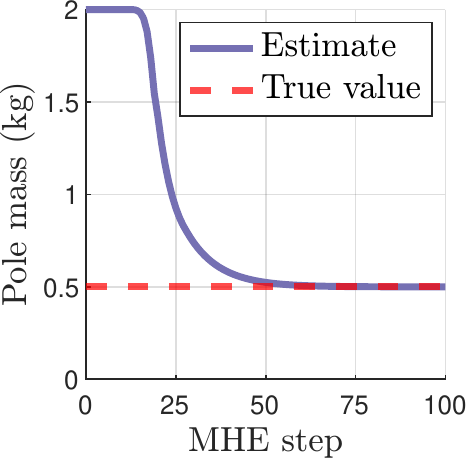}
    \caption{Pole mass estimate for the cartpole system.}
    \label{fig/cp_param_est}
\end{figure}

The cartpole is an underactuated, nonlinear dynamical system with four states and one control. \ac{PDDP} is given a bad initial estimate of the pole mass of \SI{2}{\kilogram} in its dynamics model and is tasked with finding the correct pole mass of \SI{0.5}{\kilogram} while simultaneously bringing the pole to the upright position starting from the downward configuration. The MPC and MHE horizons are chosen to be $T_\text{mpc} = T_\text{est} = 100$ steps long, while the underlying discretization step of the environment is given as $\Delta t = \SI{0.02}{\second}$. The optimization is run for 200 MPC steps.

The convergence of \ac{PDDP} to the true mass is given in \cref{fig/cp_param_est}. \ac{PDDP} makes little progress for the first 15 steps since the dynamics of the pole evolve slowly and the true mass is hard to estimate. Once the pole velocity increases, the mass can be estimated successfully within 50 steps.

\subsubsection{Quadrotor}
\begin{figure*}[h]
    \centering
    \begin{subfigure}[b]{0.25\textwidth}
        \centering
        \includegraphics[width=\textwidth]{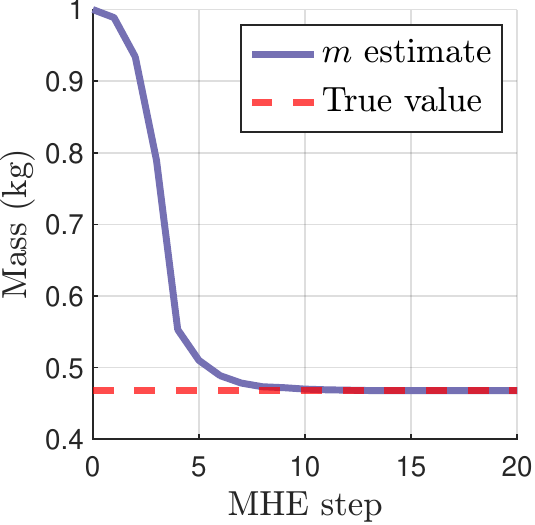}
        \caption{Mass $m$ estimate}
        \label{subfig/quad_param_m}
    \end{subfigure}%
    \begin{subfigure}[b]{0.25\textwidth}
        \centering
        \includegraphics[width=\textwidth]{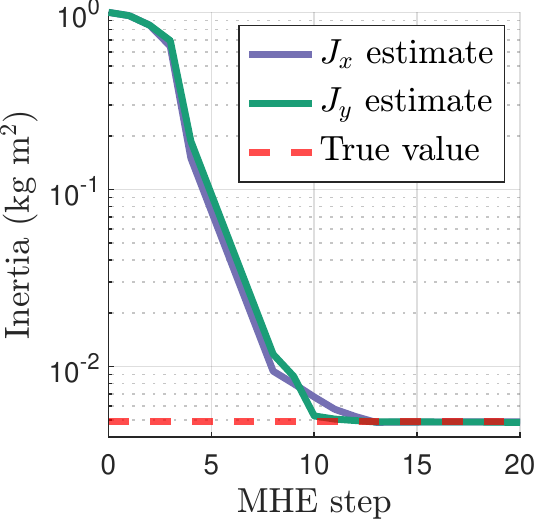}
        \caption{Inertia $J_x$ and $J_y$ estimates}
        \label{subfig/quad_param_J_x_J_y}
    \end{subfigure}%
    \begin{subfigure}[b]{0.25\textwidth}
        \centering
        \includegraphics[width=\textwidth]{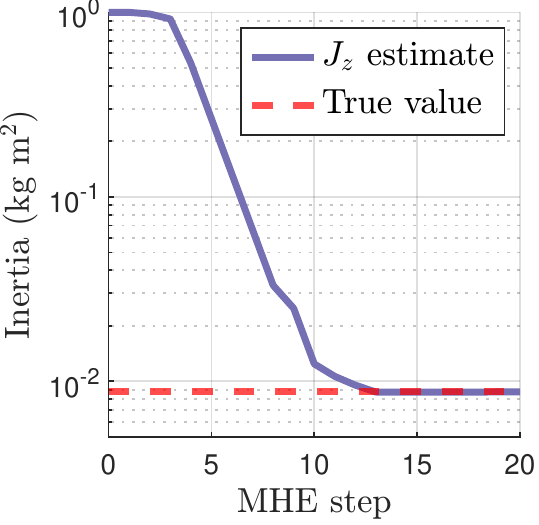}
        \caption{Inertia $J_z$ estimate}
        \label{subfig/quad_param_J_z}
    \end{subfigure}%
    \begin{subfigure}[b]{0.25\textwidth}
        \centering
        \includegraphics[width=\textwidth]{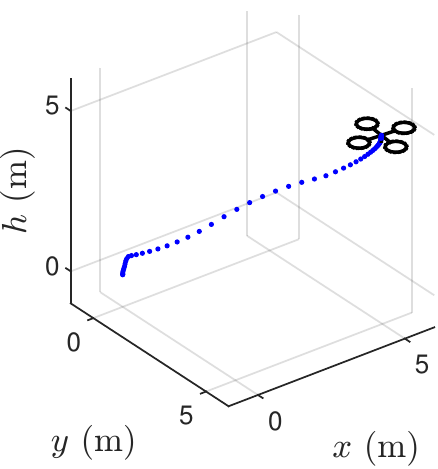}
        \caption{Quadrotor true trajectory}
        \label{subfig/quad_mhe_traj}
    \end{subfigure}
    \caption{(a-c): Parameter estimates for the quadrotor system. \ac{PDDP} converges to the true values within 15 \ac{MHE} steps. (d): The spatial trajectory of the quadrotor as it arrives at the target state. Note the slow progress at the start while the dynamics model is inaccurate.}
    \label{fig/quad_param_est}
\end{figure*}
\begin{figure*}[h]
    \centering
    \begin{subfigure}[b]{0.19\textwidth}
        \centering
        \includegraphics[width=\textwidth]{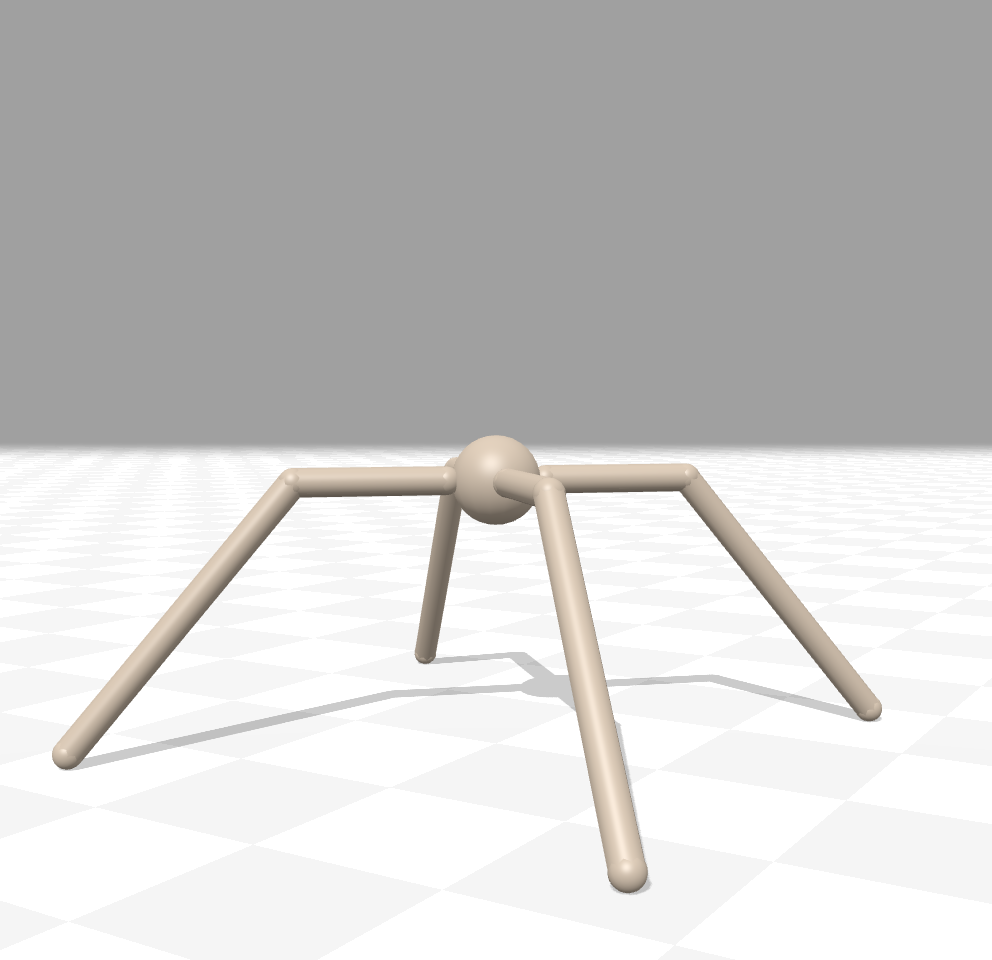}
        \caption{Initial estimate}
        \label{subfig/ant_1}
    \end{subfigure}
    \hfill
    \begin{subfigure}[b]{0.19\textwidth}
        \centering
        \includegraphics[width=\textwidth]{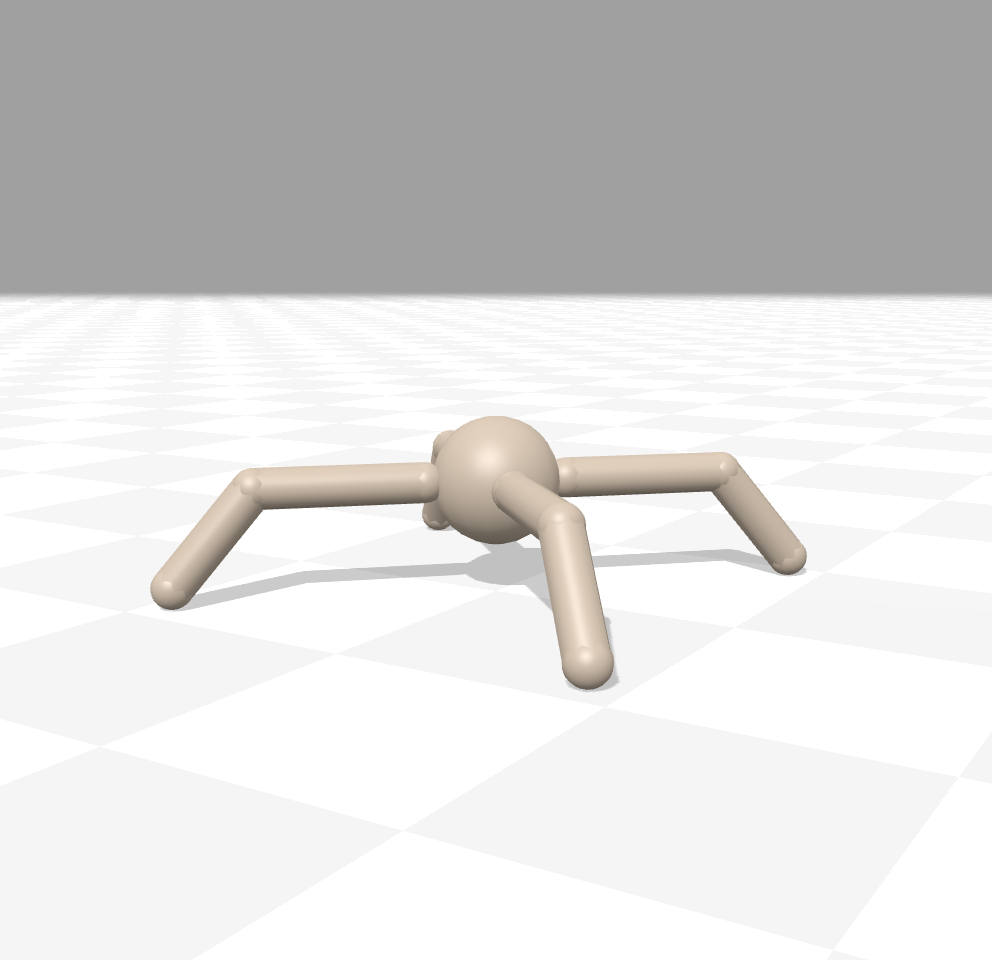}
        \caption{Estimate, \ac{MHE} step 1}
        \label{subfig/ant_2}
    \end{subfigure}
    \hfill
    \begin{subfigure}[b]{0.19\textwidth}
        \centering
        \includegraphics[width=\textwidth]{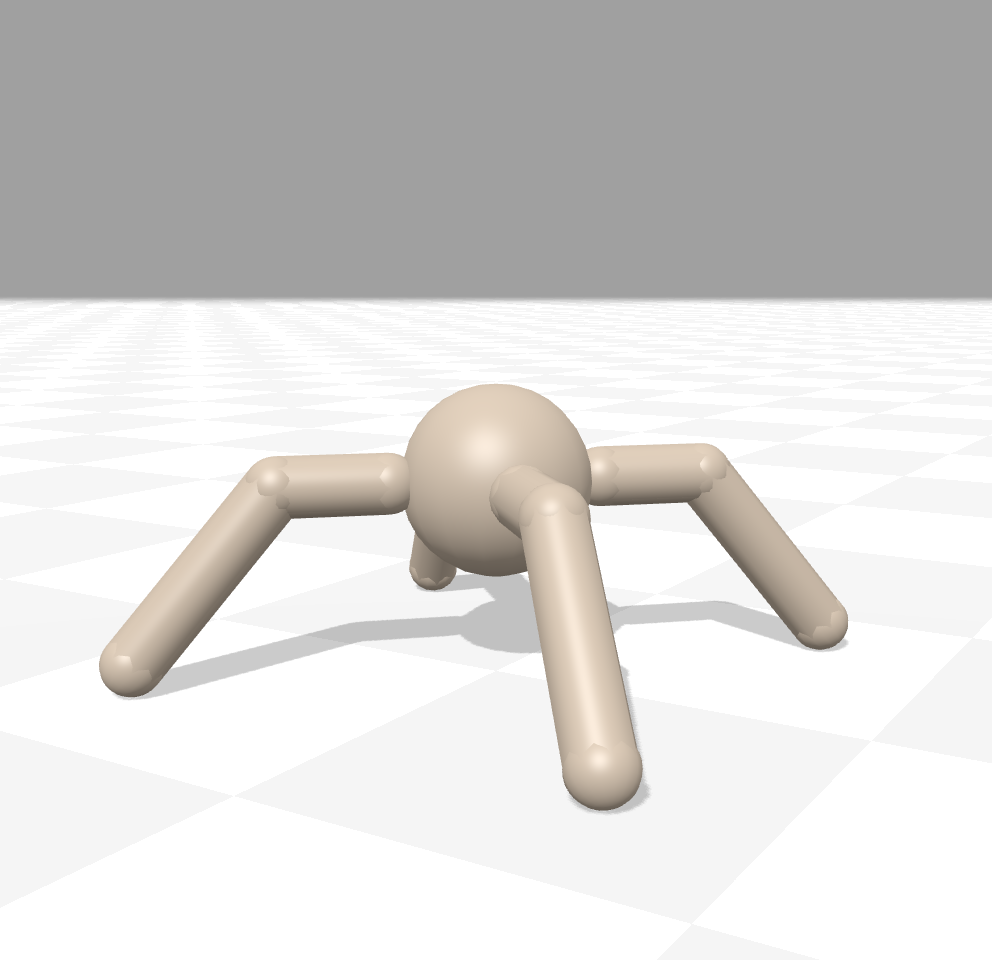}
        \caption{Estimate, \ac{MHE} step 2}
        \label{subfig/ant_3}
    \end{subfigure}%
    \hfill
    \begin{subfigure}[b]{0.2\textwidth}
        \centering
        \includegraphics[width=\textwidth]{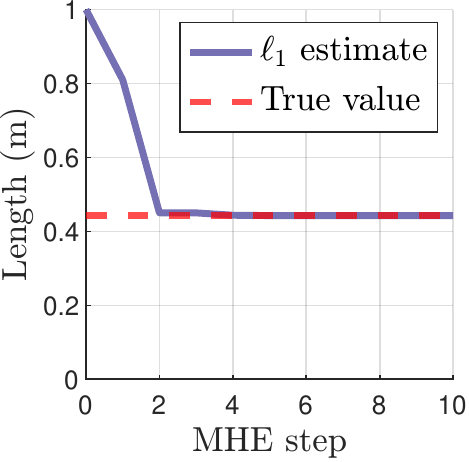}
        \caption{Upper limb length $\ell_1$}
        \label{subfig/ant_param_1}
    \end{subfigure}%
    \begin{subfigure}[b]{0.2\textwidth}
        \centering
        \includegraphics[width=\textwidth]{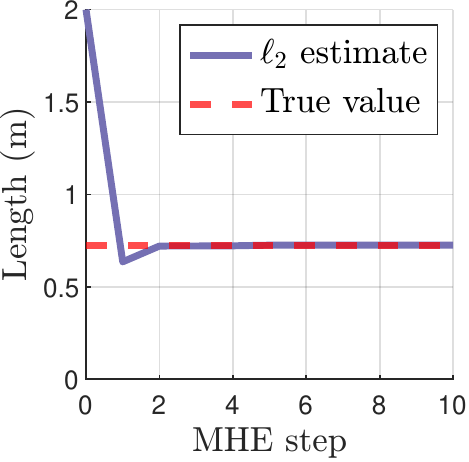}
        \caption{Lower limb length $\ell_2$}
        \label{subfig/ant_param_2}
    \end{subfigure}%
    \caption{(a-c): Leg length estimates for the ant system, rendered using Brax \cite{brax2021github}. The estimate after \ac{MHE} step 2 very closely matches the true model. (d, e): Estimate history. \ac{PDDP} converges to the true values within 4 \ac{MHE} steps.}
    \label{fig/ant_screenshots}
\end{figure*}

The quadrotor has 6 degrees of freedom, resulting in twelve states $\mathbf{x} = [x, y, h, \phi, \theta, \psi, \dot{x}, \dot{y}, \dot{h}, p, q, r]^\top$ and four controls $\mathbf{u} = [ F, \tau_\phi, \tau_\theta, \tau_\psi]^\top$ that correspond to the total thrust produced by the rotors, and the rolling, pitching, and yawing torques, respectively. The full description of the dynamics is given by \cite{beard2008quadrotor}.

\ac{PDDP} is tasked with estimating the mass and diagonal inertia terms of the quadrotor while driving the quadrotor to a target position of (5, 5, 5) \SI{}{\metre} starting from the origin. It is assumed that the quadrotor is symmetric about all three axes, so the off-diagonal terms of the inertia matrix are zero. The initial guesses of the parameters are given as $m = \SI{1}{\kilogram}$, $J_x =$~\SI{1}{\kilogram\square\metre}, $J_y =$~\SI{1}{\kilogram\square\metre}, and $J_z =$~\SI{1}{\kilogram\square\metre}, while the true parameters values are $m =$~\SI{0.468}{\kilogram}, $J_x =$~\SI{4.856e-3}{\kilogram\square\metre}, $J_y =$~\SI{4.856e-3}{\kilogram\square\metre}, and $J_z =$~\SI{8.801e-3}{\kilogram\square\metre}, which are adapted from \cite{ahmed2018sliding}. The initial estimate is very poor, and assumes the rotors are very heavy and/or far from the center of mass of the quadrotor. For this task, the MPC and MHE horizons are chosen as $T_\text{mpc} = T_\text{est} = 100$ steps, and the dynamics is integrated at a step size of $\Delta t = \SI{0.01}{\second}$.

The results are plotted in \cref{fig/quad_param_est}. \ac{PDDP} converges to the true values within 15 \ac{MHE} steps, meaning \ac{PDDP} is able to find the true parameters in under \SI{0.2}{\second} of total execution time. The \ac{MPC} task is then solved using the corrected dynamics model, taking a total of \SI{3}{\second} to arrive at the target state.

\subsubsection{Ant}



The ant is a popular robotics system for continuous control and reinforcement learning. The physics used for this system is adapted from Brax \cite{brax2021github}. This quadruped system is described by nine rigid bodies and has a state dimension of 117 in Brax, 13 states for each rigid body describing their position, quaternion rotation, linear velocity, and angular velocity. The four legs of the ant have two torque-controlled joints each, resulting in a control dimension of 8.

The \ac{MPC} task is to maximize the forward velocity of the robot, requiring periodicity in the motion of the legs to generate the necessary forward force. The \ac{MPC} horizon was chosen as $T_{\text{mpc}} = 100$ steps with an underlying environment discretization step size of $\Delta t = \SI{0.05}{\second}$, resulting in a total planning horizon of \SI{5}{\second}. The \ac{MPC} cost function rewards positive $x$ velocity and penalizes deviation of the $y$ position and total control effort.

\begin{figure*}[h]  
    \centering
    \begin{subfigure}[b]{0.3333\textwidth}
        \centering
        \includegraphics[width=\textwidth]{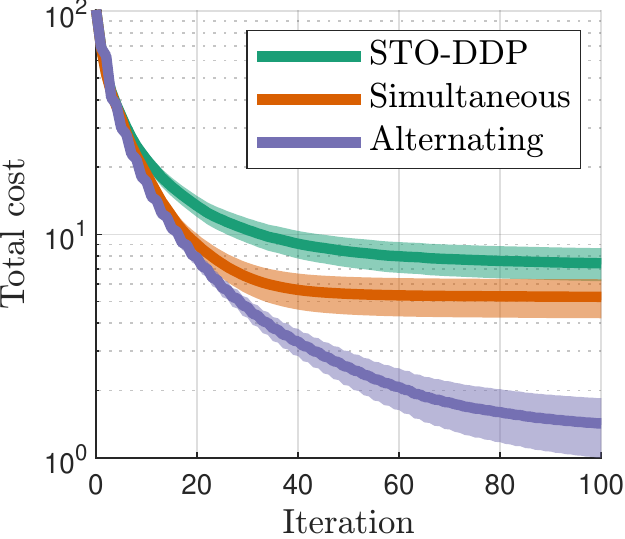}
        \caption{Cartpole}
        \label{subfig/cp_sto}
    \end{subfigure}%
    \begin{subfigure}[b]{0.3333\textwidth}
        \centering
        \includegraphics[width=\textwidth]{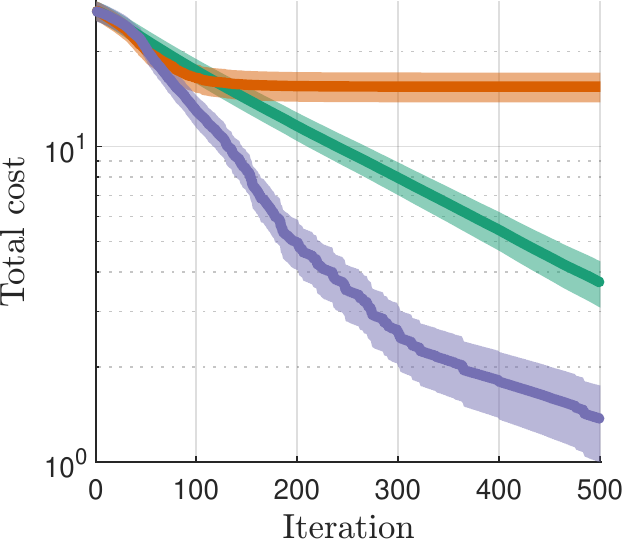}
        \caption{Quadrotor}
        \label{subfig/quad_sto}
    \end{subfigure}%
    \begin{subfigure}[b]{0.3333\textwidth}
        \centering
        \includegraphics[width=\textwidth]{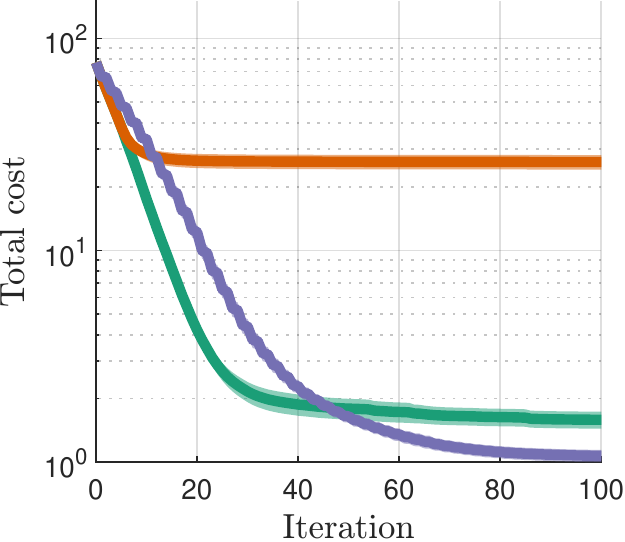}
        \caption{Lift+Cruise}
        \label{subfig/ttc_sto}
    \end{subfigure}
    \caption{Convergence of PDDP on three switching time tasks for different choices of optimization scheme. A normal distribution is fit to the cost at each iteration, with the dark lines corresponding to the mean and the shaded region showing the 95\% confidence interval. To make the comparison clearer, the total costs for each task have been normalized so the minimum cost trajectory corresponds to a cost of 1.}
    \label{fig/sto_conv}
\end{figure*}

The simultaneous \ac{MHE} task is to estimate the length of the eight rigid bodies of the legs of the ant. Symmetry is enforced to prevent instabilities, so the legs lengths are parameterized by the upper limb length $\ell_1$ and the lower limb length $\ell_2$. The true values are $\ell_1 = \SI{0.443}{\metre}$ and $\ell_2 = \SI{0.726}{\metre}$, respectively, with the initial estimate given by $\ell_1 = \SI{1.0}{\metre}$ and $\ell_2 = \SI{2.0}{\metre}$, resulting in vastly different dynamics.

The optimization is run for 250 total timesteps, with the parameter estimate history shown in \cref{subfig/ant_param_1,subfig/ant_param_2}.
\ac{PDDP} converges to the true leg lengths in 4 steps.
While the dynamics model is inaccurate, the ant makes little progress moving forward.
However, once the true parameters are found, \ac{PDDP} is able to successfully drive the ant forward.
The fast convergence to the optimal parameters is in large part due to the choice of parameterization.
Since the limb lengths $\ell_1$ and $\ell_2$ parameterize all four legs of the ant simultaneously, only a small subset of lengths can accurately describe the observed change in the state of all four legs.

\subsection{Switching Time Optimization}
\label{subsec/sto_experiments}

In this section, \ac{PDDP} is used to solve an \ac{STO} task for three systems, finding the minimum time to bring a cartpole and quadrotor to two sequential target orientations, and solving for the optimal transition point between flight regimes for an overactuated \ac{UAM} class vehicle. A numerical comparison is performed to benchmark \ac{STO-DDP} from \cite{li2020hybrid} with the simultaneous and alternating schemes proposed in \cref{subsec/alg_design}. The results are summarized in \cref{fig/sto_conv}.

For each of the following experiments, the underlying environment integration step size is chosen as $\Delta t = \SI{0.01}{\second}$ (see discussion in \cref{subsec/sto}).


\subsubsection{Multi-target cartpole}

\begin{figure}[h]
    \centering
    \includegraphics[width=0.485\textwidth]{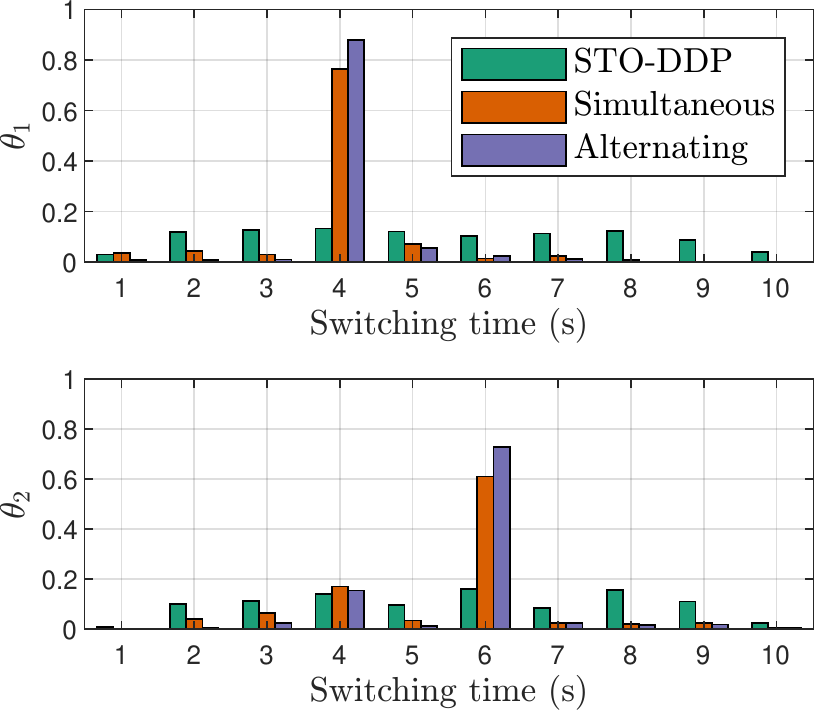}
    \caption{Histogram of optimal time scales $\theta_1$ and $\theta_2$ for the multi-target cartpole task. Bins are centered at each second from 1 up to 10. From \cref{subfig/cp_sto}, the alternating method consistently finds low cost solutions despite different initializations of the switching times, while \ac{STO-DDP} and the simultaneous method converge to suboptimal solutions.}
    \label{fig/cartpole_times}
\end{figure}


Starting from the origin with the pole in the downward position, the goal of the multi-target cartpole problem is to first bring the pole to the upright position at an $x$ position of \SI{-5}{\metre} with zero velocity, and then bring the pole to the upright position at an $x$ position of \SI{5}{\metre} with zero velocity.

In practice, the simultaneous optimization is highly dependent on the initial guess of the nominal control trajectory. Therefore, the method is warm started by only updating the controls for the first 5 iterations, corresponding to standard \ac{DDP}.

\begin{figure*}[h]
    \centering
    \includegraphics[]{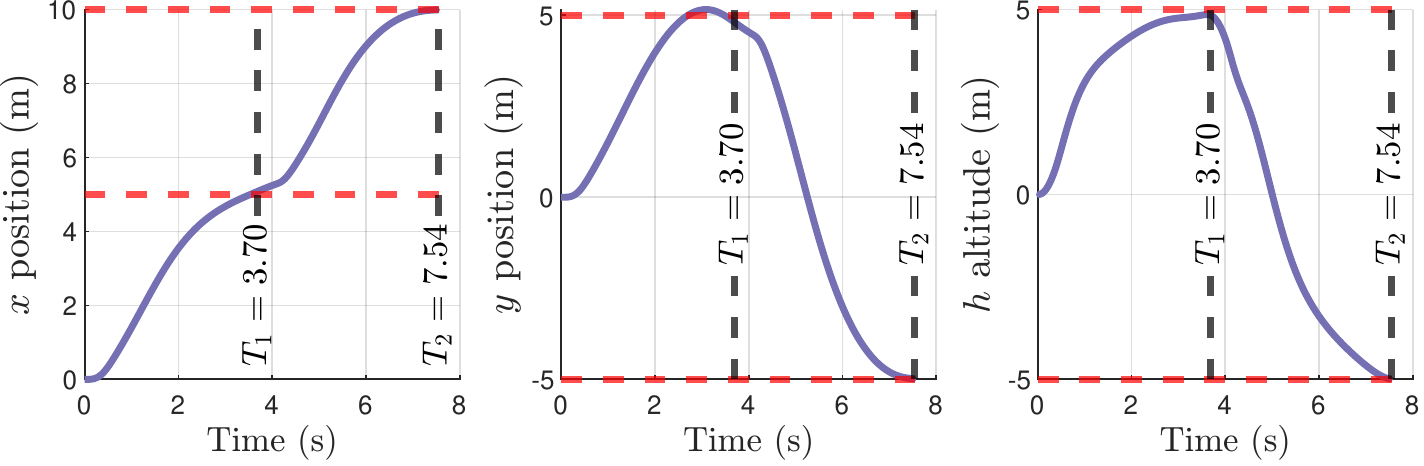}
    \caption{Time-optimal trajectory for quadrotor \ac{STO} task found using \ac{PDDP}. The terminal times of each mode $T_1$ and $T_2$ are plotted as dashed black lines. The target states are plotted as dashed red lines.}
    \label{fig/quad_sto_traj}
\end{figure*}

For each optimization scheme, \ac{PDDP} was run for 1000 initial guesses of the time scales $\theta_1$ and $\theta_2$, each sampled uniformly from the interval $[1, 10]$. The convergence behavior is plotted in \cref{subfig/cp_sto}. The alternating scheme converges to a solution half an order of magnitude lower in cost on average than both the simultaneous scheme and the \ac{STO-DDP} approach that waits for the controls to converge. Note that the simultaneous scheme reduces the cost quicker than waiting for the controls to converge, but the final solutions are comparable. Overall, the alternating scheme vastly outperforms the others on average. This success is attributed to the fact that alternating between updating the controls and updating the parameters allows the solution to more effectively escape local minima throughout the optimization process. \cref{fig/cartpole_times} shows a histogram of the optimal time scales for each of the different optimization schemes. The simultaneous and alternating scheme converge to a solution of $\theta_1 = \SI{4}{\second}$ and $\theta_2 = \SI{6}{\second}$ in the majority of cases. Waiting for the controls to converge in \ac{STO-DDP} results in a fairly flat distribution, suggesting the method quickly falls into a suboptimal solution when analyzed together with \cref{subfig/cp_sto}. Likewise, the simultaneous method converges to similar switching times as the alternating method, but does not find an optimal control solution, describing the poor convergence in \cref{subfig/cp_sto}. The alternating scheme is able to consistently converge in both switching times and controls, resulting in lower cost solutions.

\subsubsection{Quadrotor}

For the quadrotor \ac{STO} task, starting from the origin, the goal is to reach a first target position of $(5, 5, 5)$~\SI{}{\metre} with zero velocity and then to reach a second target of $(10, -5, -5)$~\SI{}{\metre} with zero velocity while optimizing for the amount of time to reach both targets.
The \ac{PDDP} algorithm was run until convergence for 1000 initial guesses of the switching times, sampled uniformly from the range [1, 10].
The convergence behavior is plotted in \cref{subfig/quad_sto}.
The simultaneous scheme quickly falls into local minima, while waiting for the controls to converge in \ac{STO-DDP} results in linear convergence.
The alternating scheme outperforms \ac{STO-DDP}.
The minimum cost trajectory is shown in \cref{fig/quad_sto_traj}.

\subsubsection{NASA Lift+Cruise vehicle}
\label{subsubsec/lpc}

\begin{figure*}[h]
    \centering
    \hfill
    \begin{subfigure}[b]{0.25\textwidth}
        \centering
        \includegraphics[width=\textwidth]{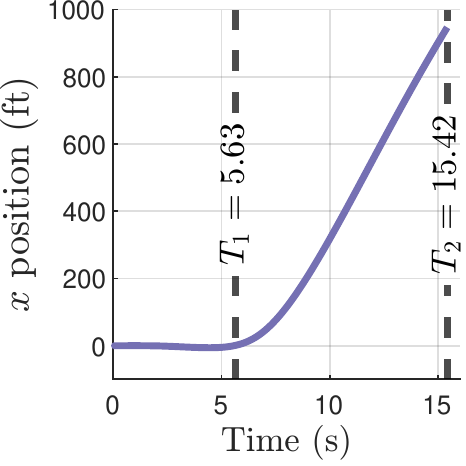}
        \caption{$x$ position trajectory.}
        \label{subfig/ttc_sto_x_traj}
    \end{subfigure}
    \begin{subfigure}[b]{0.25\textwidth}
        \centering
        \includegraphics[width=\textwidth]{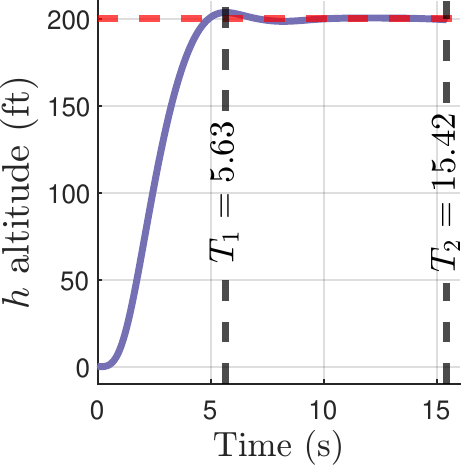}
        \caption{$h$ altitude trajectory.}
        \label{subfig/ttc_sto_h_traj}
    \end{subfigure}
    \hfill
    \begin{subfigure}[b]{0.35\textwidth}
        \centering
        \includegraphics[width=\textwidth]{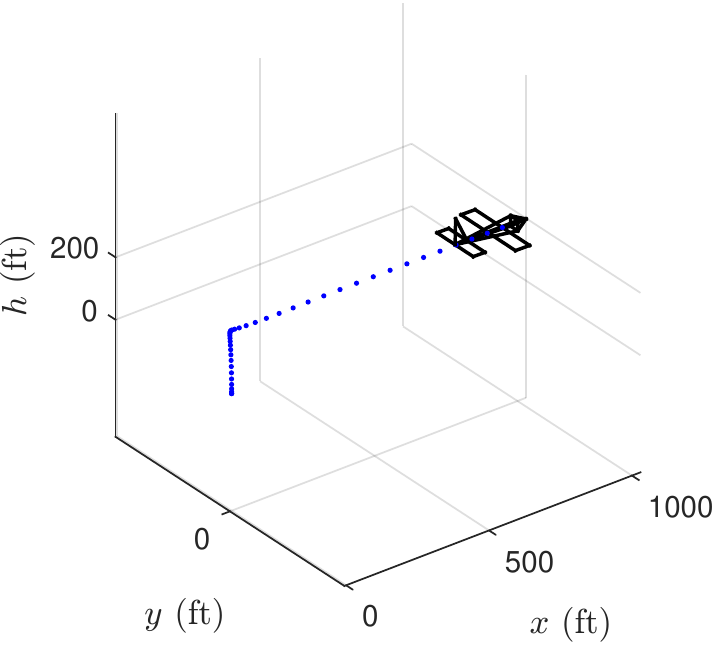}
        \caption{Spatial trajectory.}
        \label{subfig/ttc_sto_spatial_traj}
    \end{subfigure}
    \hfill
    \caption{Time-optimal vertical takeoff to cruise maneuver for the Lift+Cruise vehicle found using \ac{PDDP}. The terminal times of each mode $T_1$ and $T_2$ are plotted as dashed black lines. The target altitude is plotted as a dashed red line. \ac{PDDP} increases the altitude up to the target state during the first mode, then increases the forward velocity while maintaining altitude during the second mode.}
    \label{fig/ttc_sto}
\end{figure*}
\begin{figure*}[h]
    \centering
    \hfill
    \begin{subfigure}[b]{0.25\textwidth}
        \centering
        \includegraphics[width=\textwidth]{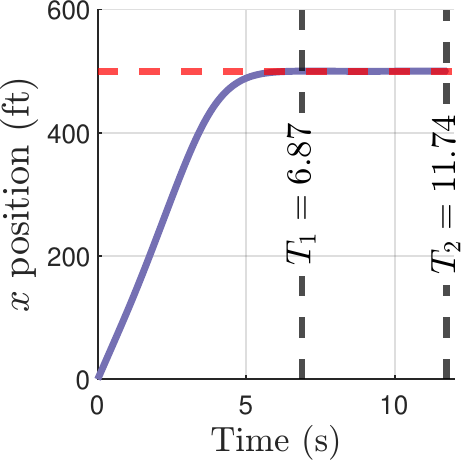}
        \caption{$x$ position trajectory.}
        \label{subfig/chl_sto_x_traj}
    \end{subfigure}
    \begin{subfigure}[b]{0.25\textwidth}
        \centering
        \includegraphics[width=\textwidth]{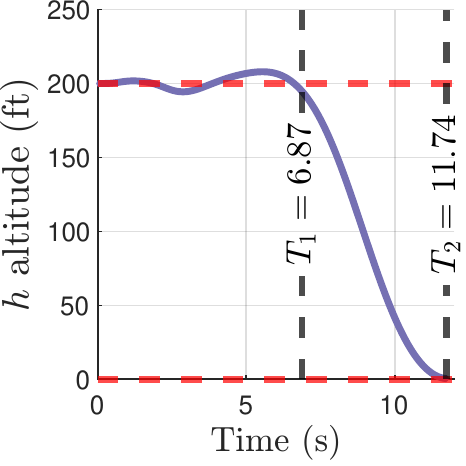}
        \caption{$h$ altitude trajectory.}
        \label{subfig/chl_sto_h_traj}
    \end{subfigure}
    \hfill
    \begin{subfigure}[b]{0.35\textwidth}
        \centering
        \includegraphics[width=\textwidth]{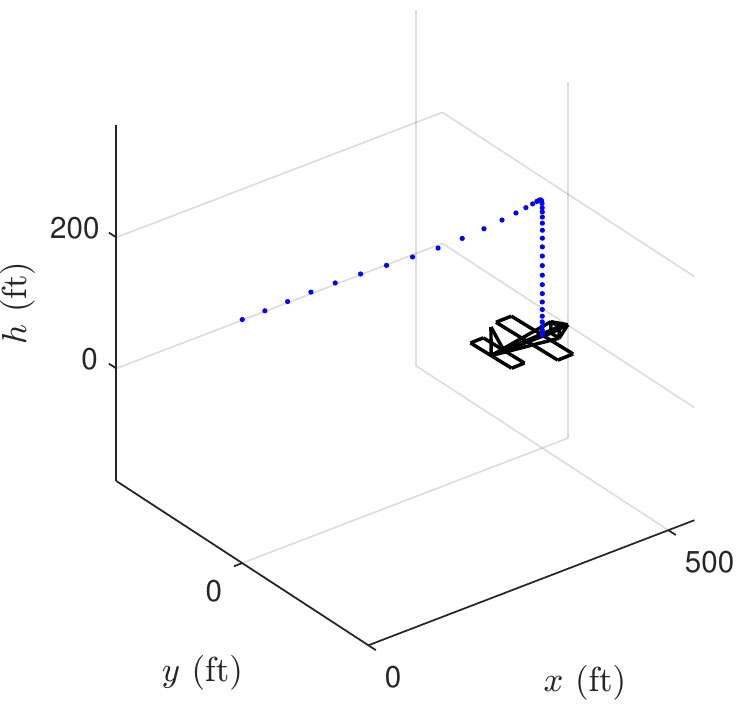}
        \caption{Spatial trajectory.}
        \label{subfig/chl_sto_spatial_traj}
    \end{subfigure}
    \hfill
    \caption{Time-optimal cruise to hover to vertical landing maneuver for the Lift+Cruise vehicle found using \ac{PDDP}. The terminal times of each mode $T_1$ and $T_2$ are plotted as dashed black lines. The target states are plotted as dashed red lines. \ac{PDDP} decreases the forward velocity of the aircraft by pitching back, then initiates a vertical landing.}
    \label{fig/chl_sto}
\end{figure*}
The NASA Lift+Cruise vehicle is a \ac{VTOL} aircraft that operates as a fixed-wing aircraft in forward flight, but has access to eight lifting rotors allowing for \ac{VTOL} capabilities.
While optimal control algorithms, including \ac{DDP}, have been applied to these vehicles in the past \cite{houghton2022path}, to the authors' best knowledge this work is the first application of switching time optimization applied to changing flight regimes within trajectory planning.
The \ac{STO} task performed here is a vertical takeoff followed by a partial transition from hover into cruise.
\ac{PDDP} was run for 500 initial guesses of the switching times sampled uniformly from [5,~15] seconds, with the convergence results of each optimization scheme plotted in \cref{subfig/ttc_sto}.
On this complex task, the simultaneous scheme converges to a poor solution.
Meanwhile, the alternating scheme and \ac{STO-DDP} result in comparable performance, with \ac{STO-DDP} beating the alternating scheme for the first 50 iterations.
The alternating scheme makes little change to the switching times while the control solution is still suboptimal during the first 50 iterations, but as the controls converge to a solution that is better able to solve the problem, the alternating scheme is able to successfully find a low cost solution for the switching times.
A representative time-optimal trajectory is shown in \cref{fig/ttc_sto}. The aircraft requires \SI{5.63}{\second} to perform the vertical takeoff maneuver and reach the desired altitude, then takes \SI{9.79}{\second} to increase its forward velocity to the desired target speed.

To further demonstrate the planning capabilities of \ac{PDDP}, the reverse partial transition maneuver was tested. This task requires the aircraft to transition from cruise to a hover configuration by bringing its forward velocity to zero within \SI{500}{ft}, and then perform a vertical landing. The solved trajectory is shown in \cref{fig/chl_sto}. This particular aircraft model does not have access to flaps or a speed brake to decelerate, and thus reduces its forward velocity by pitching back, pointing the body $z$-axis forwards to generate the backwards thrust necessary to reduce the aircraft's speed. This cruise-to-hover transition is accomplished within \SI{6.87}{\second}. The aircraft then proceeds to lower its altitude down to the target, requiring \SI{4.87}{\second} to complete the landing.

\section{Conclusions and Future Work}
\label{sec/conclusion}
This paper has derived a method generalizing previous work for solving problems with time-invariant parameters using \ac{DDP}. The proposed \ac{PDDP} algorithm is based on a general parameterized optimal control objective and allows direct optimization over time-invariant parameters with theoretical convergence guarantees. \ac{PDDP} was applied to multiple robotics systems through adaptive \ac{MPC} using \ac{MHE} and hybrid systems optimization via \ac{STO}. In particular, \ac{PDDP} is able to identify the optimal transition point between flight regimes for a \ac{UAM} class vehicle exhibiting complex transition dynamics. Various optimization schemes were analyzed and an alternating approach between updating the controls and parameters was shown to converge to a better solution by effectively escaping local minima throughout the optimization process.

In terms of runtime, a purely CPU-based MATLAB implementation of \cref{alg/pddp} with limited parallelization runs the NASA Lift+Cruise \ac{STO} task presented in \cref{subsubsec/lpc} in roughly \SI{1.5}{\second} per iteration (including the line search) on an Intel i7-9850H \SI{2.60}{\giga\hertz} processor. The runtime is dominated by the complex dynamics calculations of the NASA Lift+Cruise vehicle, especially the dynamics Jacobians calculations. The other tasks run in under a second per iteration based on the complexity of the dynamics. This is fairly slow for an MPC algorithm, and thus improvements can be made, such as improved parallelization~\cite{plancher2018performance}, in order to make the algorithm more suitable for realtime \ac{MPC} tasks.

Additionally, as described in \cite{liu2021second}, \ac{DDP} can be used as a second-order optimizer for speeding up the training process of neural networks. The addition of parameter-based optimization into this approach can help improve the training performance of modern machine learning methodologies.



\bibliographystyle{plainnat}
\bibliography{references}

\begin{thebibliography}{29}
\providecommand{\natexlab}[1]{#1}
\providecommand{\url}[1]{\texttt{#1}}
\expandafter\ifx\csname urlstyle\endcsname\relax
  \providecommand{\doi}[1]{doi: #1}\else
  \providecommand{\doi}{doi: \begingroup \urlstyle{rm}\Url}\fi

\bibitem[Adetola et~al.(2009)Adetola, DeHaan, and Guay]{adetola2009adaptive}
Veronica Adetola, Darryl DeHaan, and Martin Guay.
\newblock Adaptive model predictive control for constrained nonlinear systems.
\newblock \emph{Systems \& Control Letters}, 58\penalty0 (5):\penalty0
  320--326, 2009.

\bibitem[Ahmed and Chen(2018)]{ahmed2018sliding}
Nigar Ahmed and Mou Chen.
\newblock Sliding mode control for quadrotor with disturbance observer.
\newblock \emph{Advances in Mechanical Engineering}, 10\penalty0 (7):\penalty0
  1687814018782330, 2018.

\bibitem[Apgar et~al.(2018)Apgar, Clary, Green, Fern, and Hurst]{apgar2018fast}
Taylor Apgar, Patrick Clary, Kevin Green, Alan Fern, and Jonathan~W Hurst.
\newblock Fast online trajectory optimization for the bipedal robot {Cassie}.
\newblock In \emph{Robotics: Science and Systems}, volume 101, page~14, 2018.

\bibitem[Beard(2008)]{beard2008quadrotor}
Randal~W Beard.
\newblock Quadrotor dynamics and control.
\newblock \emph{Brigham Young University}, 19\penalty0 (3):\penalty0 46--56,
  2008.

\bibitem[Boyd and Vandenberghe(2004)]{boyd2004convex}
Stephen Boyd and Lieven Vandenberghe.
\newblock \emph{Convex Optimization}.
\newblock Cambridge University Press, 2004.

\bibitem[Branicky(1998)]{branicky1998multiple}
Michael~S Branicky.
\newblock Multiple {Lyapunov} functions and other analysis tools for switched
  and hybrid systems.
\newblock \emph{IEEE Transactions on Automatic Control}, 43\penalty0
  (4):\penalty0 475--482, 1998.

\bibitem[Di~Carlo et~al.(2018)Di~Carlo, Wensing, Katz, Bledt, and
  Kim]{di2018dynamic}
Jared Di~Carlo, Patrick~M Wensing, Benjamin Katz, Gerardo Bledt, and Sangbae
  Kim.
\newblock Dynamic locomotion in the {MIT} {Cheetah} 3 through convex
  model-predictive control.
\newblock In \emph{International Conference on Intelligent Robots and Systems
  (IROS)}, pages 1--9. IEEE, 2018.

\bibitem[Diehl et~al.(2009)Diehl, Ferreau, and Haverbeke]{Diehl2009}
Moritz Diehl, Hans~Joachim Ferreau, and Niels Haverbeke.
\newblock \emph{Efficient Numerical Methods for Nonlinear MPC and Moving
  Horizon Estimation}, pages 391--417.
\newblock Springer Berlin Heidelberg, Berlin, Heidelberg, 2009.

\bibitem[Ezzine and Haddad(1989)]{ezzine1989controllability}
Jelel Ezzine and AH~Haddad.
\newblock Controllability and observability of hybrid systems.
\newblock \emph{International Journal of Control}, 49\penalty0 (6):\penalty0
  2045--2055, 1989.

\bibitem[Freeman et~al.(2021)Freeman, Frey, Raichuk, Girgin, Mordatch, and
  Bachem]{brax2021github}
C.~Daniel Freeman, Erik Frey, Anton Raichuk, Sertan Girgin, Igor Mordatch, and
  Olivier Bachem.
\newblock Brax --- a differentiable physics engine for large scale rigid body
  simulation, 2021.
\newblock URL \url{http://github.com/google/brax}.

\bibitem[Gregory(June 14, 2021)]{gregory2021urban}
Irene~M. Gregory.
\newblock Urban air mobility: A control-centric approach to addressing
  technical challenge.
\newblock IEEE FoRCE Webinar, June 14, 2021.
\newblock URL
  \url{http://ieeecss.org/index.php/presentation/force-webinars/urban-air-mobility-control-centric-approach-addressing-technical}.

\bibitem[Gregory et~al.(2021)Gregory, Neogi, Campbell, Holbrook, Bacon, Murphy,
  Moerder, Simmons, Acheson, Britton, and Cook]{gregory2021intelligent}
Irene~M. Gregory, Natasha~A. Neogi, Newton~H. Campbell, Jon Holbrook, Barton~J.
  Bacon, Patrick~C. Murphy, Daniel~D. Moerder, Benjamin~M. Simmons, Michael~J.
  Acheson, Thomas~C. Britton, and Jacob Cook.
\newblock Intelligent contingency management for urban air mobility.
\newblock In \emph{AIAA Scitech Forum}, 2021.
\newblock \doi{10.2514/6.2021-1000}.

\bibitem[Hedlund and Rantzer(1999)]{hedlund1999optimal}
Sven Hedlund and Anders Rantzer.
\newblock Optimal control of hybrid systems.
\newblock In \emph{Proceedings of the 38th IEEE Conference on Decision and
  Control}, volume~4, pages 3972--3977. IEEE, 1999.

\bibitem[Houghton et~al.(2022)Houghton, Oshin, Acheson, Theodorou, and
  Gregory]{houghton2022path}
Matthew~D Houghton, Alexander~B Oshin, Michael~J Acheson, Evangelos~A
  Theodorou, and Irene~M Gregory.
\newblock Path planning: Differential dynamic programming and model predictive
  path integral control on {VTOL} aircraft.
\newblock In \emph{AIAA Scitech Forum}, page 0624, 2022.

\bibitem[Jacobson and Mayne(1970)]{jacobson1970differential}
David~H Jacobson and David~Q Mayne.
\newblock \emph{Differential dynamic programming}.
\newblock Number~24. Elsevier Publishing Company, 1970.

\bibitem[Kobilarov et~al.(2015)Kobilarov, Ta, and
  Dellaert]{kobilarov2015differential}
Marin Kobilarov, Duy-Nguyen Ta, and Frank Dellaert.
\newblock Differential dynamic programming for optimal estimation.
\newblock In \emph{International Conference on Robotics and Automation (ICRA)},
  pages 863--869. IEEE, 2015.

\bibitem[Li and Wensing(2020)]{li2020hybrid}
He~Li and Patrick~M Wensing.
\newblock Hybrid systems differential dynamic programming for whole-body motion
  planning of legged robots.
\newblock \emph{IEEE Robotics and Automation Letters}, 5\penalty0 (4):\penalty0
  5448--5455, 2020.

\bibitem[Li and Todorov(2004)]{li2004iterative}
Weiwei Li and Emanuel Todorov.
\newblock Iterative linear quadratic regulator design for nonlinear biological
  movement systems.
\newblock In \emph{ICINCO (1)}, pages 222--229. Citeseer, 2004.

\bibitem[Liao and Shoemaker(1992)]{liao1992advantages}
Li-zhi Liao and Christine~A Shoemaker.
\newblock Advantages of differential dynamic programming over {Newton}'s method
  for discrete-time optimal control problems.
\newblock Technical report, Cornell University, 1992.

\bibitem[Liu et~al.(2021)Liu, Chen, and Theodorou]{liu2021second}
Guan-Horng Liu, Tianrong Chen, and Evangelos Theodorou.
\newblock Second-order neural {ODE} optimizer.
\newblock \emph{Advances in Neural Information Processing Systems}, 34, 2021.

\bibitem[Mayne(1966)]{mayne1966second}
David Mayne.
\newblock A second-order gradient method for determining optimal trajectories
  of non-linear discrete-time systems.
\newblock \emph{International Journal of Control}, 3\penalty0 (1):\penalty0
  85--95, 1966.

\bibitem[Orin et~al.(2013)Orin, Goswami, and Lee]{orin2013centroidal}
David~E Orin, Ambarish Goswami, and Sung-Hee Lee.
\newblock Centroidal dynamics of a humanoid robot.
\newblock \emph{Autonomous Robots}, 35\penalty0 (2):\penalty0 161--176, 2013.

\bibitem[Plancher and Kuindersma(2018)]{plancher2018performance}
Brian Plancher and Scott Kuindersma.
\newblock A performance analysis of parallel differential dynamic programming
  on a {GPU}.
\newblock In \emph{International Workshop on the Algorithmic Foundations of
  Robotics}, pages 656--672. Springer, 2018.

\bibitem[Silva et~al.(2018)Silva, Johnson, Solis, Patterson, and
  Antcliff]{silva2018vtol}
Christopher Silva, Wayne~R Johnson, Eduardo Solis, Michael~D Patterson, and
  Kevin~R Antcliff.
\newblock {VTOL} urban air mobility concept vehicles for technology
  development.
\newblock In \emph{Aviation Technology, Integration, and Operations
  Conference}, page 3847, 2018.

\bibitem[Spong(1998)]{spong1998underactuated}
Mark~W Spong.
\newblock Underactuated mechanical systems.
\newblock In \emph{Control Problems in Robotics and Automation}, pages
  135--150. Springer, 1998.

\bibitem[Stachowicz and Theodorou(2021)]{stachowicz2021optimal}
Kyle Stachowicz and Evangelos~A Theodorou.
\newblock Optimal-horizon model-predictive control with differential dynamic
  programming.
\newblock \emph{arXiv preprint arXiv:2111.09207}, 2021.

\bibitem[Tassa et~al.(2012)Tassa, Erez, and Todorov]{tassa2012synthesis}
Yuval Tassa, Tom Erez, and Emanuel Todorov.
\newblock Synthesis and stabilization of complex behaviors through online
  trajectory optimization.
\newblock In \emph{International Conference on Intelligent Robots and Systems
  (IROS)}, pages 4906--4913. IEEE, 2012.

\bibitem[Tassa et~al.(2014)Tassa, Mansard, and Todorov]{tassa2014control}
Yuval Tassa, Nicolas Mansard, and Emo Todorov.
\newblock Control-limited differential dynamic programming.
\newblock In \emph{International Conference on Robotics and Automation (ICRA)},
  pages 1168--1175. IEEE, 2014.

\bibitem[Todorov and Li(2005)]{todorov2005generalized}
Emanuel Todorov and Weiwei Li.
\newblock A generalized iterative {LQG} method for locally-optimal feedback
  control of constrained nonlinear stochastic systems.
\newblock In \emph{Proceedings of the American Control Conference}, pages
  300--306. IEEE, 2005.

\end{thebibliography}

\clearpage
\crefalias{section}{appendix}
\crefalias{subsection}{appendix}
\appendix
\subsection{Full $Q$ derivatives}
\label{proof/full_Q_derivatives}

The full $Q$ function derivatives are given by
\begin{equation} \label{eq/full_Q_derivatives}
    \begin{split}
        & Q^0_t = \mathcal{L}^0_t + V^0_{t + 1} ,\\
        & Q^x_t = \mathcal{L}^x_t + (\mathbf{F}^x_t)^\top V^x_{t + 1} ,\\
        & Q^u_t = \mathcal{L}^u_t + (\mathbf{F}^u_t)^\top V^x_{t + 1} ,\\
        & Q^\theta_t = \mathcal{L}^\theta_t + V^\theta_{t + 1} + (\mathbf{F}^\theta_t)^\top V^x_{t + 1} ,\\
        & \begin{aligned}
            Q^{xx}_t & = \mathcal{L}^{xx}_t + (\mathbf{F}^x_t)^\top V^{xx}_{t + 1} \mathbf{F}^x_t + \textcolor{red}{V^x_{t + 1} \cdot \mathbf{F}^{xx}_t} ,
        \end{aligned} \\
        & \begin{aligned}
            Q^{xu}_t & = \mathcal{L}^{xu}_t + (\mathbf{F}^x_t)^\top V^{xx}_{t + 1} \mathbf{F}^u_t + \textcolor{red}{V^x_{t + 1} \cdot \mathbf{F}^{xu}_t} = (Q^{ux}_t)^\top ,
        \end{aligned} \\
        & \begin{aligned}
            Q^{uu}_t & = \mathcal{L}^{uu}_t + (\mathbf{F}^u_t)^\top V^{xx}_{t + 1} \mathbf{F}^u_t + \textcolor{red}{V^x_{t + 1} \cdot \mathbf{F}^{uu}_t} ,
        \end{aligned} \\
        & \begin{aligned}
            Q^{x\theta}_t = {}& \mathcal{L}^{x\theta}_t + (\mathbf{F}^x_t)^\top V^{x\theta}_{t + 1} + (\mathbf{F}^x_t)^\top V^{xx}_{t + 1} \mathbf{F}^\theta_t \\
            & + \textcolor{red}{V^x_{t + 1} \cdot \mathbf{F}^{x \theta}_t} = (Q^{\theta x}_t)^\top ,
        \end{aligned} \\
        & \begin{aligned}
            Q^{u\theta}_t = {}& \mathcal{L}^{u\theta}_t + (\mathbf{F}^u_t)^\top V^{x\theta}_{t + 1} + (\mathbf{F}^u_t)^\top V^{xx}_{t + 1} \mathbf{F}^\theta_t \\
            & + \textcolor{red}{V^x_{t + 1} \cdot \mathbf{F}^{u \theta}_t} = (Q^{\theta u}_t)^\top ,
        \end{aligned} \\
        & \begin{aligned}
            Q^{\theta\theta}_t = {}& \mathcal{L}^{\theta\theta}_t + V^{\theta\theta}_{t + 1} + V^{\theta x}_{t + 1} \mathbf{F}^{\theta}_t + (\mathbf{F}^\theta_t)^\top V^{x\theta}_{t + 1} \\
            & + (\mathbf{F}^\theta_t)^\top V^{xx}_{t + 1} \mathbf{F}^\theta_t + \textcolor{red}{V^x_{t + 1} \cdot \mathbf{F}^{\theta\theta}_t} .
        \end{aligned}
    \end{split}
\end{equation}

The terms highlighted in red denote contractions of the second-order derivative dynamics tensors with the vector $V^x_{t + 1}$. These terms are dropped in the implementation of \ac{PDDP} following \ac{iLQR} \cite{li2004iterative}.

\subsection{Proof of \cref{prop/parameters_are_newton_step}}
\label{proof/parameters_are_newton_step}
~\\
\begin{proof}
    From \cref{eq/scaled_optimal_parameter_update,eq/parameter_update_gain}, $\delta \btheta^*$ has the following form:
    \begin{equation} \label{eq/parameter_newton_step}
        \begin{split}
            \delta \btheta^* & \begin{aligned} = \epsilon \mathbf{m} \end{aligned} \\
            & \begin{aligned}
            = {}& - \epsilon (\underbrace{Q^{\theta\theta}_1 - Q^{\theta u}_1 (Q^{uu}_1)^{-1} Q^{u\theta}_1}_{V^{\theta\theta}_1})^{-1} \\
            & \times (\underbrace{Q^{\theta}_1 - Q^{\theta u}_1 (Q^{uu}_1)^{-1} Q^u_1}_{V^\theta_1})
            \end{aligned} \\
            & \begin{aligned} = - \epsilon (V^{\theta\theta}_1)^{-1} V^{\theta}_1.  \end{aligned}
    	\end{split}
    \end{equation}
    $V^{\theta}_1$ and $V^{\theta\theta}_1$ are the gradient and Hessian of the value function with respect to the parameters at the first timestep. The step in \cref{eq/parameter_newton_step} corresponds exactly to the direction given by an iteration of Newton's method for minimizing the value function at the initial time with respect to the parameters.
\end{proof}

\subsection{Proof of \cref{lemma/cost_fn_gradient}}
\label{proof/cost_fn_gradient}

\begin{proof}
    Using the expression for the cost function $\mathcal{J}$ given in \cref{eq/parameterized_cost_fn},
    \begin{equation*}
        \begin{split}
            & \begin{aligned}[b]
                \nabla_{\mathbf{u}_t} \mathcal{J} & = \nabla_{\mathbf{u}_t} \left[ \sum_{i = 1}^{T} \mathcal{L}(\mathbf{x}_i, \mathbf{u}_i; \btheta) + \phi(\mathbf{x}_{T + 1}; \btheta) \right]
            \end{aligned} \\
            & \begin{aligned}[b]
                \quad = {}& \mathcal{L}^u_t + \left( (\mathcal{L}^x_{t + 1})^\top \mathbf{F}^u_t \right)^\top + \ldots \\
                & + \left( (\mathcal{L}^x_{T})^\top \mathbf{F}^x_{T - 1} \ldots \mathbf{F}^x_{t + 1} \mathbf{F}^u_t \right)^\top \\
                & + \left( (\phi^x_{T + 1})^\top \mathbf{F}^x_{T} \mathbf{F}^x_{T - 1} \ldots \mathbf{F}^u_{t} \right)^\top
            \end{aligned} \\
            & \begin{aligned}[b]
                \quad = \mathcal{L}^u_t + (\mathbf{F}^u_t)^\top \Big( & \mathcal{L}^x_{t + 1} + (\mathbf{F}^x_{t + 1})^\top \mathcal{L}^x_{t + 2} + \ldots \\
                & + (\mathbf{F}^x_{t + 1})^\top \ldots (\mathbf{F}^x_{T - 1})^\top \mathcal{L}^x_T \\
                & + (\mathbf{F}^x_{t + 1})^\top \ldots (\mathbf{F}^x_{T})^\top \phi^x_{T + 1} \Big)
            \end{aligned} \\
            & \begin{aligned}[b]
                & \quad = \mathcal{L}^u_t + (\mathbf{F}^u_t)^\top \\
                & \times \left( \underbrace{\mathcal{L}^x_{t + 1} + (\mathbf{F}^x_{t + 1})^\top \left( \underbrace{\mathcal{L}^x_{t + 2} + \ldots + (\mathbf{F}^x_{t + 2})^\top \ldots (\mathbf{F}^x_{T})^\top \phi^x_{T + 1}}_{\eta_{t + 2}} \right)}_{\eta_{t + 1}} \right)
            \end{aligned} \\
            & \quad = \mathcal{L}^u_t + (\mathbf{F}^u_t)^\top \eta_{t + 1}
        .\end{split}
    \end{equation*}
    This shows \cref{eq/J_u} of \cref{lemma/cost_fn_gradient} is true.
\end{proof}

\subsection{Proof of \cref{lemma/updates_order_epsilon}}
\label{proof/updates_order_epsilon}

\begin{proof}
    From the definition of $\delta \btheta$ in \cref{eq/scaled_optimal_parameter_update},
    \begin{align*}
        \delta \btheta & = \epsilon \mathbf{m} = O(\epsilon)
    .\end{align*}

    Now, \cref{eq/delta_u_x_O_epsilon} is proven true by induction. Starting with $t = 1$,
    \begin{equation*}
        \begin{split}
            \delta \mathbf{u}_1 & = \epsilon \mathbf{k}_1 + \mathbf{K}_1 \underbrace{\delta \mathbf{x}_1}_{= 0} + \underbrace{\mathbf{M}_1 \delta \btheta}_{=O(\epsilon)} = O(\epsilon) , \\
            \delta \mathbf{x}_2 & = \mathbf{x}_2 - \bar{\mathbf{x}}_2 \\
            & = \mathbf{F}(\bar{\mathbf{x}}_1, \bar{\mathbf{u}}_1 + \delta \mathbf{u}_1; \bar{\btheta} + \delta \btheta) - \mathbf{F}(\bar{\mathbf{x}}_1, \bar{\mathbf{u}}_1; \bar{\btheta}) \\
            & = \mathbf{F}^u_1 \delta \mathbf{u}_1 + \mathbf{F}^\theta_1 \delta \btheta + \underbrace{O(\norm{\delta \mathbf{u}_1}_2^2 + \norm{\delta \btheta}_2^2)}_{=O(\epsilon^2)} = O(\epsilon)
        .\end{split}
    \end{equation*}
    This shows \cref{eq/delta_u_x_O_epsilon} is true for $t = 1$. Now, assume \cref{eq/delta_u_x_O_epsilon} holds for $t = i$, namely $\delta \mathbf{u}_i = O(\epsilon)$ and $\delta \mathbf{x}_{i + 1} = O(\epsilon)$. Then, for $t = i + 1$,
    \begin{equation*}
        \begin{split}
            & \delta \mathbf{u}_{i + 1} = \epsilon \mathbf{k}_{i + 1} + \mathbf{K}_{i + 1} \delta \mathbf{x}_{i + 1} + \mathbf{M}_{i + 1} \delta \btheta = O(\epsilon) ,\\
            & \begin{aligned}[b]
                \delta \mathbf{x}_{i + 2} = {}& \mathbf{F}^x_{i + 1} \delta\mathbf{x}_{i + 1} + \mathbf{F}^u_{i + 1} \delta\mathbf{u}_{i + 1} + \mathbf{F}^\theta_{i + 1} \delta \btheta \\
                & + O(\norm{\delta \mathbf{x}_{i + 1}}_2^2 + \norm{\delta \mathbf{u}_{i + 1}}_2^2 + \norm{\delta \btheta}_2^2) = O(\epsilon)
            \end{aligned}
        .\end{split}
    \end{equation*}
    Therefore, by induction, \cref{eq/delta_u_x_O_epsilon} is true for all $t = 1, \ldots, T$.
\end{proof}

\subsection{Proof of \cref{prop/control_updates_are_descent_direction}}
\label{proof/control_updates_are_descent_direction}

\begin{proof}
    To show \cref{eq/control_updates_descent_direction} is true, it is sufficient to show
    \begin{align}
        \sum_{t = 1}^{T} (\nabla_{\mathbf{u}_t} \mathcal{J})^\top \delta \mathbf{u}_t & = -\epsilon \sum_{t = 1}^{T} (\lambda_t + \gamma_t \mathbf{m}) + O(\epsilon^2) \label{eq/control_updates_descent_epsilon}
    ,\end{align}
    with $\lambda_t = (Q^u_t)^\top (Q^{uu}_t)^{-1} Q^u_t$ and $\gamma_t = (Q^u_t)^\top (Q^{uu}_t)^{-1} Q^{u \theta}_t$. To show \cref{eq/control_updates_descent_epsilon}, it is sufficient to prove using induction that
    \begin{equation} \label{eq/control_updates_inductive_step}
        \begin{split}
            \sum_{t = i}^{T} (\nabla_{\mathbf{u}_t} \mathcal{J})^\top \delta \mathbf{u}_t = {}& -\epsilon \sum_{t = i}^{T} (\lambda_t + \gamma_t \mathbf{m}) \\
            & + (V^x_i - \eta_i)^\top \delta \mathbf{x}_i + O(\epsilon^2)
        \end{split}
    .\end{equation}

    Starting at $i = T$,
    \begin{equation*}
        \begin{split}
            & (\nabla_{\mathbf{u}_T} \mathcal{J})^\top \delta \mathbf{u}_T \nonumber \\
            & \quad = (\underbrace{\mathcal{L}^u_T + (\mathbf{F}^u_T)^\top \phi^x_{T + 1}}_{Q^u_T})^\top (\epsilon \mathbf{k}_T + \mathbf{K}_T \delta \mathbf{x}_T + \mathbf{M}_T \delta \btheta) \\
            & \quad = \epsilon (Q^u_T)^\top \mathbf{k}_T + (Q^u_T)^\top \mathbf{K}_T \delta \mathbf{x}_T + (Q^u_T)^\top \mathbf{M}_T \delta \btheta \\
            & \quad \begin{aligned}[t]
                = {}& - \epsilon \underbrace{(Q^u_T)^\top (Q^{uu}_T)^{-1} Q^u_T}_{\lambda_T} - \epsilon \underbrace{(Q^u_T)^\top (Q^{uu}_T)^{-1} Q^{u\theta}_T}_{\gamma_T} \mathbf{m} \\
                & + \left[ \underbrace{Q^x_T + (\mathbf{K}_T)^\top Q^u_T}_{V^x_T} - Q^x_T \right] \delta \mathbf{x}_T
            \end{aligned} \\
            & \quad \begin{aligned}[t]
                = {}& -\epsilon (\lambda_T + \gamma_T \mathbf{m}) \\
                & + \left[ V^x_T - (\underbrace{\mathcal{L}^x_T + (\mathbf{F}^x_T)^\top \phi^x_{T + 1}}_{\eta_T}) \right]^\top \delta \mathbf{x}_T
            \end{aligned} \\
            & \quad = -\epsilon (\lambda_T + \gamma_T \mathbf{m}) + (V^x_T - \eta_T)^\top \delta \mathbf{x}_T
        .\end{split}
    \end{equation*}

    Next, assume \cref{eq/control_updates_inductive_step} holds for $i = k + 1$, namely
    \begin{equation*}
        \begin{aligned}[t]
            \sum_{t = k + 1}^{T} (\nabla_{\mathbf{u}_t} \mathcal{J})^\top \delta \mathbf{u}_t = {}& -\epsilon \sum_{t = k + 1}^{T} (\lambda_t + \gamma_t \mathbf{m}) \\
            & + (V^x_{k + 1} - \eta_{k + 1})^\top \delta \mathbf{x}_{k + 1} + O(\epsilon^2)
        .\end{aligned}
    \end{equation*}

    Then, for $i = k$,
    \begingroup
    \allowdisplaybreaks
    \begin{align*}
        & \begin{aligned}[b]
            \sum_{t = k}^T (\nabla \mathbf{u}_t \mathcal{J})^\top \delta \mathbf{u}_t = (\nabla \mathbf{u}_k \mathcal{J})^\top \delta \mathbf{u}_k + \sum_{t = k + 1}^T (\nabla \mathbf{u}_t \mathcal{J})^\top \delta \mathbf{u}_t
        \end{aligned} \\
        & \begin{aligned}[b]
            = {}& (\nabla \mathbf{u}_k \mathcal{J})^\top \delta \mathbf{u}_k - \epsilon \sum_{t = k + 1}^T (\lambda_t + \gamma_t \mathbf{m}) \\
            & + (V^x_{k + 1} - \eta_{k + 1})^\top \delta \mathbf{x}_{k + 1} + O(\epsilon^2)
        \end{aligned} \\
        & \begin{aligned}[b]
            = {}& (\mathcal{L}^u_k + (\mathbf{F}^u_k)^\top \eta_{k + 1})^\top \delta\mathbf{u}_k - \epsilon \sum_{t = k + 1}^{T} (\lambda_t + \gamma_t \mathbf{m}) \\
            & + (V^x_{k + 1} - \eta_{k + 1})^\top \Big[ \mathbf{F}^x_k \delta\mathbf{x}_k + \mathbf{F}^u_k \delta\mathbf{u}_k \\
            & \quad + \underbrace{O(\norm{\delta\mathbf{x}_k}^2 + \norm{\delta\mathbf{u}_k}^2)}_{O(\epsilon^2)} \Big] + O(\epsilon^2)
        \end{aligned} \\
        & \begin{aligned}[b]
            = {}& (\mathcal{L}^u_k + (\mathbf{F}^u_k)^\top \eta_{k + 1})^\top \delta\mathbf{u}_k - \epsilon \sum_{t = k + 1}^{T} (\lambda_t + \gamma_t \mathbf{m}) \\
            & + (V^x_{k + 1} - \eta_{k + 1})^\top \mathbf{F}^x_k \delta\mathbf{x}_k \\
            & + (V^x_{k + 1} - \eta_{k + 1})^\top \mathbf{F}^u_k \delta\mathbf{u}_k + O(\epsilon^2)
        \end{aligned} \\
        & \begin{aligned}[b]
            = {}& (\mathcal{L}^u_k + (\mathbf{F}^u_k)^\top \left[ \cancel{\eta_{k + 1} - \eta_{k + 1}} + V^x_{k + 1} \right])^\top \delta\mathbf{u}_k \\
            & - \epsilon \sum_{t = k + 1}^{T} (\lambda_t + \gamma_t \mathbf{m}) + (V^x_{k + 1} - \eta_{k + 1})^\top \mathbf{F}^x_k \delta\mathbf{x}_k + O(\epsilon^2)
        \end{aligned} \\
        & \begin{aligned}[b]
            = {}& (\underbrace{\mathcal{L}^u_k + (\mathbf{F}^u_k)^\top V^x_{k + 1}}_{Q^u_k})^\top \delta\mathbf{u}_k - \epsilon \sum_{t = k + 1}^{T} (\lambda_t + \gamma_t \mathbf{m}) \\
            & + (V^x_{k + 1} - \eta_{k + 1})^\top \mathbf{F}^x_k \delta\mathbf{x}_k + O(\epsilon^2)
        \end{aligned} \\
        & \begin{aligned}[b]
            = {}& (Q^u_k)^\top (\epsilon \mathbf{k}_k + \mathbf{K}_k \delta\mathbf{x}_k + \mathbf{M}_k \delta\btheta) \\
            & - \epsilon \sum_{t = k + 1}^{T} (\lambda_t + \gamma_t \mathbf{m}) \\
            & + (V^x_{k + 1} - \eta_{k + 1})^\top \mathbf{F}^x_k \delta\mathbf{x}_k + O(\epsilon^2)
        \end{aligned} \\
        & \begin{aligned}[b]
            = {}& \epsilon (Q^u_k)^\top \mathbf{k}_k + (Q^u_k)^\top \mathbf{K}_k \delta\mathbf{x}_k + (Q^u_k)^\top \mathbf{M}_k \delta\btheta \\
            & - \epsilon \sum_{t = k + 1}^{T} (\lambda_t + \gamma_t \mathbf{m}) \\
            & + (V^x_{k + 1} - \eta_{k + 1})^\top \mathbf{F}^x_k \delta\mathbf{x}_k + O(\epsilon^2)
        \end{aligned} \\
        & \begin{aligned}[b]
            = {}& -\epsilon \underbrace{(Q^u_k)^\top (Q^{uu}_k)^{-1} Q^u_k}_{\lambda_k} - \epsilon \underbrace{(Q^u_k)^\top (Q^{uu}_k)^{-1} Q^{u\theta}_k}_{\gamma_k} \mathbf{m} \\
            & - \epsilon \sum_{t = k + 1}^{T} (\lambda_t + \gamma_t \mathbf{m}) + (V^x_{k + 1} - \eta_{k + 1})^\top \mathbf{F}^x_k \delta\mathbf{x}_k \\
            & + (Q^u_k)^\top \mathbf{K}_k \delta\mathbf{x}_k + O(\epsilon^2)
        \end{aligned} \\
        & \begin{aligned}[b]
            = {}& - \epsilon \sum_{t = k}^{T} (\lambda_t + \gamma_t \mathbf{m}) + \Big[ \mathcal{L}^x_k \underbrace{-\mathcal{L}^x_k - (\mathbf{F}^x_k)^\top \eta_{k + 1}}_{-\eta_k} \\
            & \qquad + (\mathbf{F}^x_k)^\top V^x_{k + 1} + \mathbf{K}_k^\top Q^u_k \Big]^\top \delta\mathbf{x}_k + O(\epsilon^2)
        \end{aligned} \\
        & \begin{aligned}[b]
            = {}& - \epsilon \sum_{t = k}^{T} (\lambda_t + \gamma_t \mathbf{m}) \\
            & + \Big[ \underbrace{\mathcal{L}^x_k + (\mathbf{F}^x_k)^\top V^x_{k + 1}}_{Q^x_k} + \mathbf{K}_k^\top Q^u_k - \eta_k \Big]^\top \delta\mathbf{x}_k + O(\epsilon^2)
        \end{aligned} \\
        & \begin{aligned}[b]
            = {}& - \epsilon \sum_{t = k}^{T} (\lambda_t + \gamma_t \mathbf{m}) \\
            & + \Big[ \underbrace{Q^x_k + \mathbf{K}_k^\top Q^u_k}_{V^x_k} - \eta_k \Big]^\top \delta\mathbf{x}_k + O(\epsilon^2)
        \end{aligned} \\
        & = - \epsilon \sum_{t = k}^{T} (\lambda_t + \gamma_t \mathbf{m}) + (V^x_k - \eta_k)^\top \delta\mathbf{x}_k + O(\epsilon^2)
    \end{align*}
    \endgroup

    This implies, by induction, that \cref{eq/control_updates_inductive_step} holds for all $i = T, \ldots, 1$. Taking $i = 1$ in \cref{eq/control_updates_inductive_step} implies \cref{eq/control_updates_descent_epsilon} is true, since $\delta \mathbf{x}_1 = 0$.

    Finally, by \cref{prop/parameters_are_newton_step}, the form of \cref{eq/control_updates_descent_epsilon} can be expressed as
    \begin{equation*}
        \begin{split}
            \mkern-18mu \sum_{t = 1}^{T} (\nabla_{\mathbf{u}_t} \mathcal{J})^\top \delta \mathbf{u}_t & = -\epsilon \sum_{t = 1}^{T} (\lambda_t + \gamma_t \mathbf{m}) + O(\epsilon^2) \\
            & = - \epsilon \sum_{t = 1}^{T} \lambda_t + \epsilon \left( \sum_{t = 1}^{T} \gamma_t \right) (V^{\theta\theta}_1)^{-1} V^\theta_1 + O(\epsilon^2)
        .\end{split}
    \end{equation*}
    This shows that there exists some $\epsilon$ sufficiently small such that \cref{eq/control_updates_descent_direction} holds.
\end{proof}

\subsection{Proof of \cref{lemma/value_fn_full_zero_order_term}}
\label{proof/value_fn_full_zero_order_term}

\begin{proof}
    Starting with \cref{eq/parameterized_Q_approximation} and substituting in the expressions for $\delta \mathbf{u}_t^*$ from \cref{eq/optimal_control_update} and $\delta \btheta^*$ from \cref{eq/scaled_optimal_parameter_update} yields
    \begin{equation*}
        \begin{split}
            & \begin{aligned}[b]
                & Q(\mathbf{x}_t, \mathbf{u}_t^*; \btheta^*) \\
                & \quad \approx \bigg[ Q^0_t + \epsilon (Q^u_t)^\top \mathbf{k}_t + \epsilon (Q^u_t)^\top \mathbf{M}_t \mathbf{m} + \epsilon (Q^\theta_t)^\top \mathbf{m} \\
                & \qquad + \frac{1}{2} \epsilon^2 \mathbf{k}_t^\top Q^{uu}_t \mathbf{k}_t + \epsilon^2 \mathbf{k}_t^\top Q^{uu}_t \mathbf{M}_t \mathbf{m} \\
                & \qquad + \frac{1}{2} \epsilon^2 \mathbf{m}^\top \mathbf{M}_t^\top Q^{uu}_t \mathbf{M}_t \mathbf{m} + \epsilon^2 \mathbf{k}_t Q^{u\theta}_t \mathbf{m} \\
                & \qquad + \epsilon^2 \mathbf{m}^\top \mathbf{M}_t^\top Q^{u\theta}_t \mathbf{m} + \frac{1}{2} \epsilon^2 \mathbf{m}^\top Q^{\theta\theta}_t \mathbf{m} \bigg] \\
                & \qquad + \bigg[ (Q^x_t)^\top + (Q^u_t)^\top \mathbf{K}_t + \epsilon \mathbf{k}_t^\top Q^{ux}_t \\
                & \qquad + \epsilon \mathbf{m}^\top \mathbf{M}_t^\top Q^{ux}_t + \epsilon \mathbf{m}^\top Q^{\theta x}_t + \epsilon \mathbf{k}_t^\top Q^{uu}_t \mathbf{K}_t \\
                & \qquad + \epsilon \mathbf{m}^\top \mathbf{M}^\top Q^{uu}_t \mathbf{K}_t + \epsilon \mathbf{m}^\top Q^{\theta u}_t \mathbf{K}_t \bigg] \delta \mathbf{x}_t \\
                & \qquad + \frac{1}{2} \delta \mathbf{x}_t^\top \bigg[ Q^{xx}_t + 2 Q^{xu}_t \mathbf{K}_t + \mathbf{K}_t^\top Q^{uu}_t \mathbf{K}_t \bigg] \delta \mathbf{x}_t
            ,\end{aligned}
        \end{split}
    \end{equation*}
    where the zero-, first-, and second-order terms in $\delta \mathbf{x}_t$ have been grouped together. Equating like powers in \cref{eq/parameterized_value_fn_approx} gives
    \begingroup
    \allowdisplaybreaks
    \begin{align*}
        & \begin{aligned}[b]
            & V^0_t = Q^0_t + \epsilon (Q^u_t)^\top \mathbf{k}_t + \epsilon (Q^u_t)^\top \mathbf{M}_t \mathbf{m} + \epsilon (Q^\theta_t)^\top \mathbf{m} \\
            & \qquad + \frac{1}{2} \epsilon^2 \mathbf{k}_t^\top Q^{uu}_t \mathbf{k}_t + \epsilon^2 \mathbf{k}_t^\top Q^{uu}_t \mathbf{M}_t \mathbf{m} \\
            & \qquad + \frac{1}{2} \epsilon^2 \mathbf{m}^\top \mathbf{M}_t^\top Q^{uu}_t \mathbf{M}_t \mathbf{m} + \epsilon^2 \mathbf{k}_t Q^{u\theta}_t \mathbf{m} \\
            & \qquad + \epsilon^2 \mathbf{m}^\top \mathbf{M}_t^\top Q^{u\theta}_t \mathbf{m} + \frac{1}{2} \epsilon^2 \mathbf{m}^\top Q^{\theta\theta}_t \mathbf{m}
        \end{aligned} \\
        & \begin{aligned}[b]
            & \quad = Q^0_t - \epsilon \underbrace{(Q^u_t)^\top (Q^{uu}_t)^{-1} Q^u_t}_{\lambda_t} - \epsilon (Q^u_t)^\top (Q^{uu}_t)^{-1} Q^{u\theta}_t \mathbf{m} \\
            & \qquad + \epsilon (Q^\theta_t)^\top \mathbf{m} + \frac{1}{2} \epsilon^2 \underbrace{(Q^u_t)^\top (Q^{uu}_t)^{-1} Q^u_t}_{\lambda_t} \\
            & \qquad + \epsilon^2 (Q^u_t)^\top (Q^{uu}_t)^{-1} Q^{u\theta}_t \mathbf{m} + \frac{1}{2} \epsilon^2 \mathbf{m}^\top Q^{\theta u}_t (Q^{uu}_t)^{-1} Q^{u\theta}_t \mathbf{m} \\
            & \qquad - \epsilon^2 (Q^u_t)^\top (Q^{uu}_t)^{-1} Q^{u\theta}_t \mathbf{m} - \epsilon^2 \mathbf{m}^\top Q^{\theta u}_t (Q^{uu}_t)^{-1} Q^{u\theta}_t \mathbf{m} \\
            & \qquad + \frac{1}{2} \epsilon^2 \mathbf{m}^\top Q^{\theta\theta}_t \mathbf{m}
        \end{aligned} \\
        & \begin{aligned}[b]
            & \quad = Q^0_t - \epsilon (1 - \frac{1}{2} \epsilon) \lambda_t + \epsilon (\underbrace{Q^\theta_t - Q^{\theta u}_t (Q^{uu}_t)^{-1} Q^u_t}_{V^\theta_t} )^\top \mathbf{m} \\
            & \qquad + \frac{1}{2} \epsilon^2 \mathbf{m}^\top (\underbrace{Q^{\theta\theta}_t - Q^{\theta u}_t (Q^{uu}_t)^{-1} Q^{u \theta}_t}_{V^{\theta\theta}_t}) \mathbf{m}
        ,\end{aligned}
    \end{align*}
    \endgroup
    which shows the form of \cref{eq/V_0_t_full}. Now let $t = 1$ and substituting in the expression for $\mathbf{m}$ given in \cref{eq/parameter_newton_step} yields
    \begin{equation*}
        \begin{split}
            & \begin{aligned}[b]
                & V^0_1 = Q^0_1 - \epsilon (1 - \frac{1}{2} \epsilon) \lambda_1 \\
                & \qquad - \epsilon \underbrace{(V^\theta_1)^\top (V^{\theta\theta}_1)^{-1} V^\theta_1}_{\psi} + \frac{1}{2} \epsilon^2 \underbrace{(V^\theta_1)^\top (V^{\theta\theta}_1)^{-1} V^\theta_1}_{\psi}
            \end{aligned} \\
            & \quad = Q^0_1 - \epsilon (1 - \frac{1}{2} \epsilon) (\lambda_1 + \psi)
        ,\end{split}
    \end{equation*}
    which shows \cref{eq/V_0_1}.
\end{proof}

\subsection{Proof of \cref{prop/cost_reduction}}
\label{proof/cost_reduction}

\begin{proof}
    Let the cost of the nominal trajectory $\bar{\mathbf{x}}_t, \bar{\mathbf{u}}_t, \bar{\btheta}$ after iteration $k$ be $\mathcal{J}^{(k)}$. The cost $\mathcal{J}^{(k + 1)}$ after iteration $k + 1$ when applying the optimal control and parameter update is given by $V^0_1$ given in \cref{eq/V_0_1} plus higher-order terms, namely
    \begin{equation*}
        \begin{split}
            \mathcal{J}^{(k + 1)} & = V^0_1 + O(\epsilon^3) \\
            & = Q^0_1 - \epsilon (1 - \frac{1}{2} \epsilon) (\lambda_1 + \psi) + O(\epsilon^3) \\
            & = \mathcal{L}^0_1 + V^0_2 - \epsilon (1 - \frac{1}{2} \epsilon) (\lambda_1 + \psi) + O(\epsilon^3) \\
            & = \mathcal{L}^0_1 + Q^0_2 - \epsilon (1 - \frac{1}{2} \epsilon) \left( \sum_{t = 1}^2 \lambda_t + \psi \right) + O(\epsilon^3) \\
            & = \sum_{t = 1}^2 \mathcal{L}_t^0 + V_3^0 - \epsilon (1 - \frac{1}{2} \epsilon) \left( \sum_{t = 1}^2 \lambda_t + \psi \right) + O(\epsilon^3) \\
            & \quad \vdots \nonumber \\
            & = \underbrace{\sum_{t = 1}^{T} \mathcal{L}_t^0 + \phi_{T + 1}^0}_{\mathcal{J}^{(k)}} - \epsilon (1 - \frac{1}{2} \epsilon) \left( \sum_{t = 1}^T \lambda_t + \psi \right) + O(\epsilon^3) \\
            & = \mathcal{J}^{(k)} - \epsilon (1 - \frac{1}{2} \epsilon) \left( \sum_{t = 1}^T \lambda_t + \psi \right) + O(\epsilon^3)
        .\end{split}
    \end{equation*}
    Thus,
    \begin{equation*}
        \begin{split}
            \Delta \mathcal{J}^{(k + 1)} & = \mathcal{J}^{(k + 1)} - \mathcal{J}^{(k)} \\
            & = - \epsilon (1 - \frac{1}{2} \epsilon) \left( \sum_{t = 1}^T \lambda_t + \psi \right) + O(\epsilon^3)
        .\end{split}
    \end{equation*}
\end{proof}

\newpage
\subsection{Proof of \cref{thm/pddp_converges}}
\label{proof/pddp_converges}

\begin{proof}
    First, note that for $0 < \epsilon \leq 1$,
    \begin{equation*}
        \begin{split}
            -(1 - \frac{1}{2} \epsilon) \leq -\frac{1}{2} \implies -\epsilon (1 - \frac{1}{2} \epsilon) \leq -\frac{1}{2} \epsilon
        .\end{split}
    \end{equation*}
    Therefore, by \cref{prop/cost_reduction}, the cost reduction after the $k^\text{th}$ iteration is upper bounded by
    \begin{align}
        \Delta \mathcal{J}^{(k)} & \leq - \frac{1}{2} \epsilon \left(\sum_{t = 1}^T \lambda_t + \psi \right) \label{eq/cost_reduction_upper_bound}
    .\end{align}

    Now, by \cref{prop/parameters_are_newton_step} and \cref{prop/control_updates_are_descent_direction}, there exists some $0 < \epsilon_1 \leq 1$ such that for all $0 < \epsilon \leq \epsilon_1$,
    \begin{equation*}
        \begin{split}
            \sum_{t = 1}^T (\nabla_{\mathbf{u}_t} \mathcal{J})^\top \delta \mathbf{u}_t & \leq - \epsilon \sum_{t = 1}^T \lambda_t + \epsilon \left( \sum_{t = 1}^T \gamma_t \right) (V^{\theta\theta}_1)^{-1} V^\theta_1 \\
            \delta \btheta & = -\epsilon (V^{\theta\theta}_1)^{-1} V^\theta_1
        .\end{split}
    \end{equation*}
    Likewise, by \cref{eq/cost_reduction_upper_bound}, there exists some $0 < \epsilon_2 \leq \epsilon_1$ such that for all $0 < \epsilon \leq \epsilon_2$,
    \begin{equation*}
        \Delta \mathcal{J}^{(k)} \leq - \frac{1}{2} \epsilon \left( \sum_{t = 1}^T \lambda_t + \psi \right)
    ,\end{equation*}
    which implies the sequence of costs after successive iterations of \ac{PDDP} is monotonically decreasing.

    To proceed, it is assumed that the space of controls and parameters is compact. Thus, since $\mathcal{J}$ is a continuous function over a compact space, it is bounded, so there exists some $\mathbf{U}^*$ and $\btheta^*$ such that
    \begin{align}
        \lim_{k \to \infty} \mathcal{J}^{(k)} & = \mathcal{J}(\mathbf{U}^*; \btheta^*) \label{eq/costs_converge}
    ,\end{align}
    meaning the sequence of costs converge to a minimum. Finally, it can be shown that $\mathbf{U}^{(k)}$ and $\btheta^{(k)}$ converge to the minimizers.

    Note that \cref{eq/costs_converge} implies that $\Delta \mathcal{J}^{(k)} \to 0$ as $k \to \infty$. This further implies that for all $t = 1, \ldots, T$, $\lambda_t \to 0$ and $\psi \to 0$; so
    \begin{equation*}
        \begin{split}
            \lambda_t \to 0 & \implies (Q^u_t)^\top (Q^{uu}_t)^{-1} Q^u_t \to 0 \\
            & \implies (Q^{uu}_t)^{-1} Q^u_t = \mathbf{k}_t \to \mathbf{0} \\
            \psi \to 0 & \implies (V^\theta_1)^\top (V^{\theta\theta}_1)^{-1} V^\theta_1 \to 0 \\
            & \implies (V^{\theta\theta}_1)^{-1} V^\theta_1 = \mathbf{m} \to \mathbf{0}
        .\end{split}
    \end{equation*}
    Recall $\delta \btheta = \epsilon \mathbf{m}$, so $\delta \btheta \to \mathbf{0}$ and $\btheta^{(k)} \to \btheta^*$ as $k \to \infty$. Using induction, it is proven that for all $t = 1, \ldots, T$,
    \begin{equation} \label{eq/u_x_updates_converge}
        \begin{split}
            \delta \mathbf{u}_t & \to \mathbf{0} \\
            \delta \mathbf{x}_{t + 1} & \to \mathbf{0}
        ,\end{split}
    \end{equation}
    as $k \to \infty$.

    At time $t = 1$, since the initial condition is fixed, $\delta \mathbf{x}_1 = \mathbf{0}$; thus
    \begin{equation*}
        \begin{split}
            \delta \mathbf{u}_1 & = \epsilon \mathbf{k}_1 + \mathbf{M}_1 \delta \btheta \to \mathbf{0} \\
            \delta \mathbf{x}_2 & = \mathbf{F}^x_1 \delta \mathbf{x}_1 + \mathbf{F}^u_1 \delta \mathbf{u}_1 + \mathbf{F}^\theta_1 \delta \btheta \\
            & \quad + O(\norm{\delta \mathbf{x}_1}_2^2 + \norm{\delta \mathbf{u}_1}_2^2 + \norm{\delta \btheta}_2^2) \to \mathbf{0}
        .\end{split}
    \end{equation*}
    Now, assume \cref{eq/u_x_updates_converge} holds for $t = i$, namely
    \begin{equation*}
        \begin{split}
            \delta \mathbf{u}_i & \to \mathbf{0} \\
            \delta \mathbf{x}_{i + 1} & \to \mathbf{0}
        .\end{split}
    \end{equation*}
    Then, for time $t = i + 1$,
    \begin{equation*}
        \begin{split}
            & \delta \mathbf{u}_{i + 1} = \epsilon \mathbf{k}_{i + 1} + \mathbf{K}_{i + 1} \delta \mathbf{x}_{i + 1} + \mathbf{M}_{i + 1} \delta \btheta \to \mathbf{0} \\
            & \begin{aligned}[b]
                \delta \mathbf{x}_{i + 2} = {}& \mathbf{F}^x_{i + 1} \delta \mathbf{x}_{i + 1} + \mathbf{F}^u_{i + 1} \delta \mathbf{u}_{i + 1} + \mathbf{F}^\theta_{i + 1} \delta \btheta \\
                & + O(\norm{\delta \mathbf{x}_{i + 1}}_2^2 + \norm{\delta \mathbf{u}_{i + 1}}_2^2 + \norm{\delta \btheta}_2^2) \to \mathbf{0}
            .\end{aligned}
        \end{split}
    \end{equation*}
    This proves \cref{eq/u_x_updates_converge} holds by induction, implying $\mathbf{U}^{(k)} \to \mathbf{U}^*$.
\end{proof}

\end{document}